\documentclass[10pt]{amsart}
\usepackage{amsmath,amscd}
\usepackage{graphicx}
\usepackage[all]{xy}
\usepackage{amsthm}
\usepackage{amssymb,url,extarrows}
\usepackage[bookmarks=true]{hyperref}
\usepackage{color}
\usepackage{subfigure}

\usepackage{booktabs}
\usepackage{tikz}
\usepackage{tabularx}
\usetikzlibrary{arrows}
\usetikzlibrary{shapes.geometric}

\newtheorem{thm}{Theorem}[section]
\newtheorem{thm*}{Theorem}
\newtheorem{lem}[thm]{Lemma}
\newtheorem{cor}[thm]{Corollary}
\newtheorem*{conj}{Leinster-Willerton conjecture}
\newtheorem{prop}[thm]{Proposition}

\theoremstyle{remark}
\newtheorem{remark}[thm]{Remark}
\newtheorem{example}[thm]{Example}

\theoremstyle{definition}
\newtheorem{deef}[thm]{Definition}
\newtheorem{problem}{Problem}

\newcommand{\R}{\mathbb{R}}
\newcommand{\C}{\mathbb{C}}

\newcommand{\N}{\mathbb{N}}

\newcommand{\rd}{\mathrm{d}}

\renewcommand{\epsilon}{\varepsilon}

\newcommand{\e}{\mathrm{e}}

\title{The magnitude and spectral geometry}
\author{Heiko Gimperlein, Magnus Goffeng, Nikoletta Louca}
\address{Heiko Gimperlein\newline
\indent Leopold-Franzens-Universit\"{a}t Innsbruck\newline
\indent Technikerstraße 13 \newline
\indent 6020 Innsbruck\newline
\indent Austria \newline
\newline
\indent Nikoletta Louca\newline
\indent Maxwell Institute for Mathematical Sciences and \newline
\indent Department of Mathematics, Heriot-Watt University\newline
\indent Edinburgh EH14 4AS\newline
\indent United Kingdom\newline
\newline
\indent Magnus Goffeng,\newline
\indent Centre for Mathematical Sciences\newline 
\indent Lund University\newline 
\indent Box 118, SE-221 00 Lund\newline 
\indent Sweden\newline}
\subjclass[2010]{}
\keywords{51F99 (primary), 58J50, 52A39, 58J40, 31C12 (secondary)}
\email{heiko.gimperlein@uibk.ac.at, nl24@hw.ac.uk, magnus.goffeng@math.lth.se}

\begin{document}
\maketitle

\begin{abstract}
We study the geometric significance of Leinster's notion of magnitude for a smooth manifold with boundary of arbitrary dimension, motivated by open questions for the unit disk in $\mathbb{R}^2$. For a large class of distance functions, including embedded submanifolds of Euclidean space and Riemannian manifolds satisfying a technical condition, we show that the magnitude function is well defined for $R\gg 0$ and admits a meromorphic continuation to sectors in $\C$. We obtain an asymptotic variant of the convex magnitude conjecture by Leinster and Willerton: In the limit $R \to \infty$ the magnitude function admits an asymptotic expansion, which determines the volume, surface area and integrals of generalized curvatures. Lower-order terms are computed by black box computer algebra. As a consequence, we initiate the study of magnitude analogues to classical questions in spectral geometry. 
\end{abstract}

\section{Introduction}

The notion of \emph{magnitude} of an enriched category, and specifically for a compact metric space, was introduced by Leinster \cite{leinster} to capture the ``essential size'' of an object. The magnitude has been shown to generalize the cardinality of a set and the Euler characteristic, and it is even closely related to measures of the diversity of a biological system. See \cite{leinmeck} for an overview.

For a finite metric space $(X,\rd)$, we say that $w : X \to \R$ is a weight function provided that $\sum_{y \in X} \mathrm{e}^{-\rd(x,y)}w(y) =1$ for all $x \in X$. Given a weight function $w$, the magnitude of a finite metric space $X$ is defined as $\mathrm{mag}(X) := \sum_{x \in X} w(x)$. The magnitude $\mathrm{mag}(X)$ is independent of the choice of $w$. More generally, for a compact, positive definite metric space $(X,\rd)$ as in \cite{meckes1}, the magnitude is defined as
\begin{equation}\label{supdefinition}\mathrm{mag}(X) := \sup\{\mathrm{mag}(\Xi) : \Xi \subset X\ \ \text{finite}\}\ .\end{equation}

This article finds a geometric origin of magnitude for smooth, compact manifolds with boundary $X$ with a suitable distance function $\rd$. In particular, it provides a framework for the analysis and explicit computations of magnitude when $X \subset \mathbb{R}^n$ is a compact domain, a long-standing problem when $n$ is even. The framework combines classical results in semiclassical analysis with current results for pseudodifferential boundary problems. As an application, we initiate the study of magnitude analogues to classical questions in spectral geometry and prove an asymptotic variant of the Leinster-Willerton conjecture.\\

Instead of the magnitude of an individual space $(X,\rd)$, it proves fruitful to study the function $\mathcal{M}_X(R) := \mathrm{mag}(X,R\cdot \rd)$ for $R>0$. 
For a compact, convex subset $X\subset \mathbb{R}^n$, Leinster and Willerton \cite{leinwill} conjectured a surprising relation between magnitude and classical objects in convex geometry, the intrinsic volumes $V_k(X)$:

\begin{conj}
Suppose $X\subset \mathbb{R}^n$ is compact and convex. Then 
\begin{align*}
\mathcal{M}_X(R) =& \frac{1}{n! \omega_n} \mathrm{vol}_n(X)\ R^n +  \frac{1}{2(n-1)! \omega_{n-1}} \mathrm{vol}_{n-1}(\partial X)\ R^{n-1} + \cdots + 1 \\ 
=& \sum_{k=0}^n\frac{1}{k! \omega_k} V_k(X)\ R^k \ .
\end{align*}
Here, $\omega_k$ is the volume of the $k$-dimensional unit ball. 
\end{conj}
This formula would imply an inclusion-exclusion principle $\mathcal{M}_X + \mathcal{M}_Y = \mathcal{M}_{X\cup Y}+\mathcal{M}_{X\cap Y}$ for magnitude, a fundamental property of  the Euler characteristic and the other intrinsic volumes. See \cite{barcarbs, leinwill} for further motivation. 

The Leinster-Willerton conjecture is easily verified for an interval $X \subset \mathbb{R}$, where $\mathcal{M}_X(R) = \frac{\mathrm{Length}(X)}{2} R + \chi(X)$. More generally, it holds for convex bodies $X\subset \mathbb{R}^n$ with the $\ell_1$, instead of the Euclidean norm \cite{leinmeck}. 
In dimension $n=5$, Barcel\'o and Carbery \cite{barcarbs} gave a counterexample to the original formulation of the Leinster-Willerton conjecture: $\mathcal{M}_{B_5}$ is a rational function, not a polynomial. In odd dimensions their work related magnitude to a differential boundary problem, which allowed two of the authors \cite{gimpgoff} to prove a corrected, asymptotic formulation of the conjecture. It identified classical geometric invariants beyond $V_k(X)$, which are encoded by $\mathcal{M}_{X}$ when $X \subset \mathbb{R}^n$ is a compact, smooth domain, $n$ odd: The magnitude function $\mathcal{M}_X$ extends to a meromorphic function and admits an asymptotic expansion $$\mathcal{M}_X(R) \sim \frac{1}{n! \omega_n} \sum_{j=0}^\infty c_j(X)\ R^{n-j}\ \  \text{ for $R \to +\infty,$}$$ in the sense that for any $N\in \N$, 
$\mathcal{M}_X(R)-\sum_{j=0}^N c_j(X)R^{n-j} =\mathcal{O}(R^{n-N-1})$ as $R\to +\infty$. Furthermore, for $j=0,1,2$ there exist $\gamma_{n,j} \in \mathbb{Q}$ independent of $X$ such that $c_j(X) = \gamma_{n,j} V_j(X)$ for convex $X$. The fourth term, $c_3(X)$, is by the work  \cite{gimpgoff4th} proportional to the Willmore energy $\int_{\partial X} H^2$, a geometric quantity not predicted by the Leinster-Willerton conjecture.

In this article we extend the results of \cite{gimpgoff} to arbitrary dimensions and to geometric settings, using a new, unified approach. For the unit disk $B_2 \subset \mathbb{R}^2$ we address long-standing questions about the form of $\mathcal{M}_{B_2}$: We find 
$$\mathcal{M}_{B_2}(R) = \frac{1}{2}R^2 + \frac{3}{2}R  + \frac{9}{8}+O(R^{-1}),$$ 
and additionally that $\mathcal{M}_{B_2}(R)$ is not a polynomial. More generally, our approach leads to an algorithm to calculate higher $c_j(X)$, which has been implemented in python. We here state the result in the simplest setting of smooth, compact, planar domains.

\begin{thm}
\label{diskthm}
Let $X \subset \mathbb{R}^2$ be a smooth, compact domain.
\begin{enumerate}
\item[a)] $\mathcal{M}_X$ admits a meromorphic continuation to $\mathbb{C}\setminus \{0\}$.
\item[b)] There exists an asymptotic expansion
\begin{equation*}
\mathcal{M}_X(R)\ \sim \ \frac{1}{2\pi} \sum_{j=0}^\infty c_j(X) R^{2-j} \quad \mbox{as $R \to +\infty$}.
\end{equation*}
\item[c)] The first three coefficients are given by
\begin{align*}
&c_0(X)=\textnormal{Area}(X),\  c_1(X)=\frac{3}{2} \textnormal{Perim}(\partial X),\  c_2(X)=\frac{9}{8} \int_{\partial X} H\, \rd S\ ,
\end{align*} 
\end{enumerate}
where $H$ is the mean curvature and $\textnormal{Perim}(\partial X)=\int_{\partial X}  \rd S$ denotes the perimeter of the boundary. 
\end{thm}

\begin{remark}
\label{computersays}
The computer code described in Appendix \ref{appecode} leads to $c_3(X) ={\gamma} \int_{\partial X} H^2\, \rd S$, $c_4(X) = {\delta} \int_{\partial X} H^3 \, \rd S$, where $0\neq \gamma, \delta \in \mathbb{Q}$ are constants independent of $X$. In particular, for a smooth, compact domain $X \subset \mathbb{R}^2$, the magnitude function $\mathcal{M}_X$ is not a polynomial. The code is available at \href{http://www.hw.ac.uk/~hg94}{this link}.

The fact that $c_j(X)\propto \int_{\partial X} H^{j-1}\rd S$ for $j=1,2,3$ in odd dimensions (see \cite{gimpgoff4th}), and for $j=1,2,3,4$ in dimension $n=2$, makes it natural to ask if $c_j(X)\propto \int_{\partial X} H^{j-1}\rd S$ for any $n>1$, $j>0$ and any smooth domain $X\subseteq \R^n$. 
\end{remark}

Theorem \ref{diskthm} is a special case of Theorem \ref{Rntheorem}, stated below for Euclidean domains. 
Subsection \ref{subseconasus} discusses extensions of Theorem \ref{diskthm} to manifolds with boundary, with geometric implications the content of Section \ref{sec:geometry}.

The starting point for this work is a reformulation of magnitude by Willerton \cite{will} and Meckes \cite{meckes1, meckes}. Willerton \cite{will} extended the notion of a weight vector to the integral operator ${\mathcal{Z}}_X(R)$ on the space $M(X)$ of finite Borel measures on $X$:  
$$\mathcal{Z}_X(R):M(X)\to C(X), \quad \mathcal{Z}_X(R)\mu(x):=\frac{1}{R}\int_X\mathrm{e}^{-R\mathrm{d}(x,y)}\mathrm{d}\mu(y)\ .$$
A \emph{weight measure} $\mu_R$ is a solution to the equation 
$$R\mathcal{Z}_X(R)\mu_R=1\ .$$ 
If $(X,R\cdot\mathrm{d})$ is positive definite and admits a weight measure, Meckes \cite{meckes1} showed that $\mathcal{M}_X(R)=\mu_R(X)$.
We shall show that for a domain $X\subset \mathbb{R}^n$ the operator $\mathcal{Z}_X(R)$ is a pseudodifferential operator and therefore extends to the Sobolev space $\dot{H}^{-\frac{n+1}{2}}(X)$ of distributions supported in $X$. The equation $R\mathcal{Z}_X(R) u_R =1$ on $X$ admits a unique distributional solution $u_R\in \dot{H}^{-\frac{n+1}{2}}(X)$. By relating this approach to Meckes's work \cite{meckes}, we conclude that $\mathcal{M}_X(R)=\langle u_R,1\rangle_X$.

This pseudodifferential framework allows to use methods from semiclassical analysis to study $\mathcal{M}_X$. A key ingredient is the construction of an approximate inverse to $\mathcal{Z}_X$, based on recent advances for pseudodifferential boundary problems \cite{g3}. The approach extends to subsets of manifolds with a distance function, subject to appropriate technical assumptions.

We note that after the completion of this article related questions have been addressed for generic finite metric spaces \cite{ohararec,roffyoshi}.

\subsection{Acknowledgements}
We thank Magnus Fries, Daniel Grieser, Gerd Grubb and Mark Meckes for useful comments on this article and Tony Carbery, Daniel Grieser, Gerd Grubb, Tom Leinster, Rafe Mazzeo, Mark Meckes, Niels Martin M\o ller, Grigori Rozenblum, Jan-Philip Solovej and Simon Willerton for fruitful discussions.

MG was supported by the Swedish Research Council Grant VR 2018-0350. NL was supported by The Maxwell Institute Graduate School in Analysis and its Applications, a Centre for Doctoral Training funded by the UK Engineering and Physical Sciences Research Council (Grant EP/L016508/01), the Scottish Funding Council, Heriot-Watt University and the University of Edinburgh.

\section{Overview and main results}
\label{sec:resandz}

Let $X$ be  a compact manifold with boundary equipped with a distance function $\rd$ and a volume density $\rd y$. Meckes's abstract capacity-theoretic approach to the  magnitude function $\mathcal{M}_X$ in \cite{meckes} relies on the family of integral operators
\begin{align}\mathcal{Z}_X(R)u(x):=\frac{1}{R} \int_X \e^{-R\rd(x,y)}u(y)\rd y, \quad R\in \C\setminus \{0\},\label{magequaq}
\end{align}
to study $\mathcal{M}_X(R)=R^{-1}\langle \mathcal{Z}_X(R)^{-1}1,1\rangle_X$. Explicit calculations of magnitude have avoided the solution of this integral equation, based on reformulations available for certain homogeneous manifolds \cite{will} and odd-dimensional Euclidean domains \cite{barcarbs,gimpgoff,gimpgoff4th,willoddb}. 

In this article we study the magnitude function in the general setting of a manifold $X$ with boundary endowed with a distance function $\rd$, by making Meckes's approach explicit. We require that $\rd^2(x,y)$ is a smooth function in a small neighborhood of the diagonal $x=y$. The operator $\mathcal{Z}_X(R)$ then turns out to be a parameter-dependent pseudodifferential operator on $X$, up to an error term which is often negligible. This allows to adapt methods from semiclassical analysis and recent developments for pseudodifferential boundary value problems to study  the inverse $\mathcal{Z}_X(R)^{-1}$ and the magnitude function $\mathcal{M}_X(R)=R^{-1}\langle \mathcal{Z}_X(R)^{-1}1,1\rangle_X$ in terms of the geometry of $(X,\rd)$.

The following theorem illustrates our results for the geometric content of the magnitude function. We here state them for the geodesic distance function of a Riemannian metric. See Remark \ref{mrandsmr} or \cite[Section 3]{gimpgofflouc} for a discussion of the properties (MR) and (SMR). 

\begin{thm}
\label{introthm}
Let $X$ be a compact $n$-dimensional Riemannian manifold with boundary equipped with its geodesic distance $\rd$. Let $g$ denote its Riemannian metric and $\rd x$ its associated density. Assume that the geodesic distance has property $(MR)$. Then there exists $R_0> 0$ such that $(X, R\cdot \rd)$ is a positive definite metric space for all $R>R_0$. Its magnitude function  $\mathcal{M}_X$ admits an asymptotic expansion
$$\mathcal{M}_X(R)\ \sim \ \frac{1}{n!\omega_n}\sum_{j=0}^\infty c_j(X) R^{n-j},$$
as $R\to \infty$. The geometric content and structure of the coefficients $c_j$ are as follows:
\begin{itemize}
\item $c_0(X)=\mathrm{vol}_n(X,g)$ 
\item $c_1(X)=\frac{n+1}{2}\mathrm{vol}_{n-1}(\partial X,g)$
\item If $n\geq 2$, $c_2(X)=\frac{n+1}{6}\int_X s\mathrm{d}x+\frac{(n+1)^2(n-1)}{8}\int_{\partial X} H\mathrm{d}S$, where $s$ is the scalar curvature and $H$ the mean curvature of the boundary.
\item For $j\geq 4$, the coefficient $c_j(X)=c_j(X^\circ)+c_j(X,\partial X)$ where $c_j(X^\circ )$ can be computed as an integral over $X$ of a universal polynomial in covariant derivatives of the curvature of total degree $\leq j$ and $c_j(X,\partial X)$ is an integral over $\partial X$ of a universal polynomial in covariant derivatives of the curvature  as well as the second fundamental form of $\partial X$ of total degree $\leq j$. Moreover, $c_j(X^\circ )=0$ for odd $j$.
\end{itemize}
If the distance function has property (SMR) on a sector $\Gamma\subseteq \C$, e.g.~if $\rd^2$ is smooth and $\Gamma=\C\setminus \{0\}$, then $\mathcal{M}_X$ extends meromorphically to $\Gamma$.
\end{thm}

Theorem \ref{introthm} is a special case of Theorem \ref{symbcorboundary} and Theorem \ref{technicalthmsec6}, which includes extensions to non-Riemannian distance functions. Non-Riemannian distance functions are of interest, for example, for submanifolds of $\mathbb{R}^n$ with the subspace distance \cite{will}. For Euclidean domains and the original setting of the Leinster-Willerton conjecture, additional information is obtained in Theorem \ref{Rntheorem}.\\

Geometric applications are discussed in Section \ref{sec:geometry}, motivated by classical questions in spectral geometry and the category-theoretic orgin of magnitude. For instance, we conclude the following results:
\begin{enumerate}
\item The magnitude recovers the Euler characteristic of surfaces, see Subsection \ref{eulerrel}.
\item The integral form of $c_j$  implies an inclusion-exclusion principle for smooth, compact domains $X,Y,X\cap Y$ in a manifold which has property $(MR)$ (see Subsection \ref{inadiandin}): 
$$\mathcal{M}_{X\cup Y}(R) \sim \mathcal{M}_X(R)+\mathcal{M}_Y(R)-\mathcal{M}_{X\cap Y}(R).$$
\item Analogous to Kac' famous question \cite{kac}, one can ``magnitude the shape of a drum'' for balls: if $X$ is a smooth, compact domain in $\R^n$, $B$ is a ball and $\mathcal{M}_X(R)=\mathcal{M}_{B}(R)+o(R^{n-1})$, then $X$ is isometric to $B$. See Subsection \ref{specgeo1}.
\item Analogous to a theorem by Berger \cite{berger}, magnitude characterizes Euclidean domains whose boundary has constant mean curvature when the dimension $n$ is odd or $n=2$: if $X$ and $Y$ are smooth, compact, $n$-dimensional domains with $\mathcal{M}_X(R)=\mathcal{M}_{Y}(R)+o(R^{n-2})$, then $\partial X$ has constant mean curvature if and only if $\partial Y$ has constant mean curvature. See Subsection \ref{specgeo2}.\\
\end{enumerate}

Let us illustrate Theorem \ref{introthm} and its extensions by asymptotic computations of the magnitude function in classical Riemannian and non-Riemannian geometries. Beyond the intrinsic interest \cite{barcarbs, gimpgoff, leinwill, will, willoddb}, these computations shed light on continuity properties of $X \mapsto\mathcal{M}_X$, as raised by the Leinster-Willerton conjecture, and on the technical assumption (MR).

\begin{example}[Cylinders]
\label{cylexample}
Suppose that $M\subseteq \R^N$ is a compact $n-1$-dimensional submanifold, $T>0$ a parameter and $X_T=M\times [0,T]\subset \R^{N+1}$ is the cylinder on $M$ of height $T$. The space $X_T$ is a compact, $n$-dimensional manifold with boundary, depicted in Figure \ref{plot}(a). Equipping $X_T$ with the subspace distance, Proposition \ref{cylcompmpmad} relates the magnitude of $X_T$ with the magnitude of the base $M$: 
$$\frac{n!\omega_n\mathcal{M}_{X_T}(R)}{(n-1)!\omega_{n-1}\mathcal{M}_{M}(R)}=TR+\frac{n+1}{2}+T\cdot O(1), \quad\mbox{as $R\to \infty$},$$
and in particular
$$n!\omega_{n}\mathcal{M}_{X_T}(R)=\mathrm{vol}_{n-1}(M)\left(TR^n+\frac{n+1}{2}R^{n-1}\right)+T\cdot O(R^{n-2}), \quad\mbox{as $R\to \infty$}.$$ Lower-order terms can be expressed in terms of the geometry of $M$ alone.
\end{example}

\begin{figure}
   \centering
   \subfigure[]{
   \includegraphics[height=4.1cm]{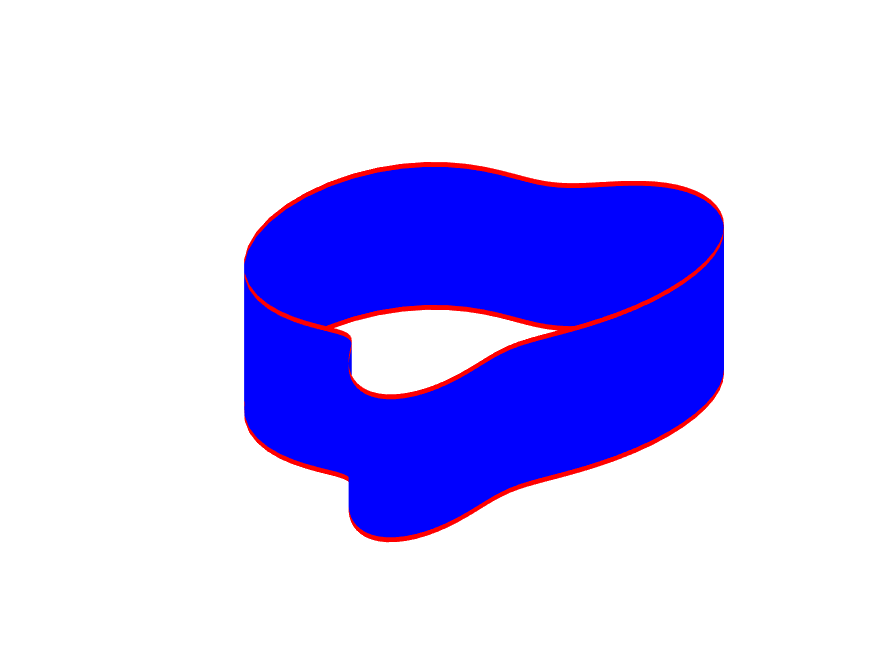}}
   \subfigure[]{
   \includegraphics[height=3cm]{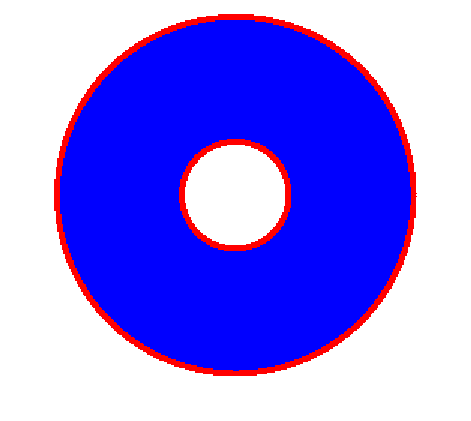}}\\
 \subfigure[]{
  \includegraphics[height=3cm]{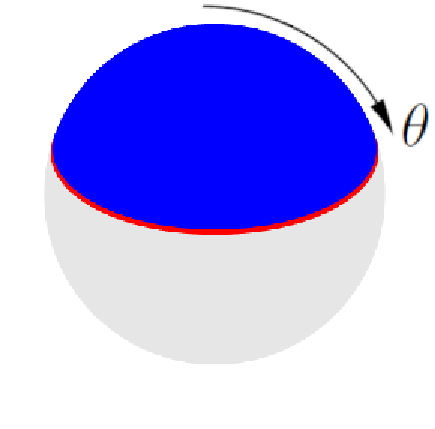}}\hspace*{0.4cm}
\subfigure[]{
  \includegraphics[height=2.5cm]{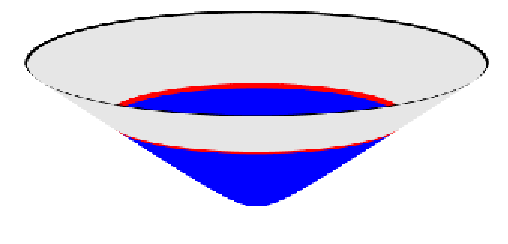}}\hspace*{0.2cm}
\subfigure[]{
  \includegraphics[height=2.7cm]{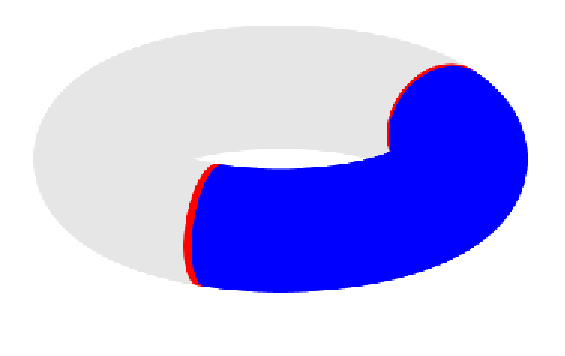}}
  
 \caption{(a) Cylinder, Example \ref{cylexample}, (b) spherical shell, Example \ref{shellexample},\\ (c) spherical cap, Example \ref{capexample}, (d) ball in hyperbolic plane (hyperboloid model), Example \ref{hyperexample}, (e) toroidal armband, Example \ref{partoftorii}.}
\label{plot}
\end{figure}

\begin{example}[Spherical shells]
\label{shellexample}
For $\epsilon\in [0,1]$ and $n=2,3$ consider the family of spherical shells
$$X_\epsilon:=\{x\in \mathbb{R}^n: 1-\epsilon\leq |x|\leq 1+\epsilon\},$$
depicted in Figure \ref{plot}(b). The limit case $\epsilon=0$ corresponds to the unit circle or sphere, $\epsilon=1$ to the ball of radius $2$. 

In dimension $n=2$, Theorem \ref{introthm} shows $$\mathcal{M}_{X_\epsilon}(R) = 2 \epsilon R^2 + 3 R + O(R^{-1})$$
when $\epsilon \in(0,1)$, while $\mathcal{M}_{X_1}(R) = 2R^2 + 3R + \frac{9}{8}+ O(R^{-1})$ and $\mathcal{M}_{X_0}(R) = \frac{\pi R}{1-e^{-\pi R}}\sim \pi R$. In the limit $\epsilon\to 1$, one observes continuity for the coefficients of $R^2$ and of $R$, but the $R^0$ term in the expansion jumps. For $\epsilon\to 0$, the coefficients of $R^2$ and of $R^0$ are continuous, but the coefficient of $R^1$ jumps.

This may be contrasted with the $n=3$-dimensional case, where the full magnitude function can be computed: 
$$\mathcal{M}_{X_\epsilon}(R)=\begin{cases} 
\frac{2R^2+2}{1-\mathrm{e}^{-\pi R}},  &\epsilon=0,\\
\frac{2\epsilon^2+6\epsilon}{3!}R^3+(2\epsilon^2+2)R^2+4\epsilon R+2+\quad&\\
\quad +\frac{\mathrm{e}^{-R(1-\epsilon)}(R^2(1-\epsilon)^2+1)+2R^3(1-\epsilon)^3-3R^2(1-\epsilon)^2+2R(1-\epsilon)-1}{\sinh(2R(1-\epsilon))-2R(1-\epsilon)}, \;&\epsilon\in (0,1)\\
\frac{8}{3!}R^3+4R^2+4R+1, &\epsilon=1. \end{cases}$$
See \cite[Example 36]{gimpgoff} for the very similar computations for a spherical shell $(2B_3)\setminus B_3^\circ$, and \cite{will} for $X_0 = S^2$. Note that 
$$\lim_{\epsilon\to 1}\mathcal{M}_{X_\epsilon}(R)=\mathcal{M}_{X_1}(R)\ ,$$
$$\lim_{\epsilon\to 0}\mathcal{M}_{X_\epsilon}(R)=2R^2+2=\mathcal{M}_{X_0}(R) +O(R^{-\infty}).$$ 
Still, the $R^0$ term in the expansion here is the Euler characteristic of $\chi(X_\epsilon)$, which jumps when the topology changes. For $\epsilon\to 0$ there is a jump of dimension, and the pointwise values correspondingly have a jump. 

Because the family $X_\epsilon$ is continuous in the Hausdorff distance, we conclude that the magnitude function is not continuous with respect to this topology, unlike for convex domains. Nevertheless, the asymptotics takes the same form as $R\to \infty$.
\end{example}

\begin{example}[Sphere]
\label{capexample}
Let $\theta\in [0,\pi]$ and $X_\theta\subseteq S^2$ denote the smooth, compact spherical cap depicted in Figure \ref{plot}(c), 
$$X_\theta := \{(x,y,z)\in S^2\subseteq \R^3 : z\geq \cos(\theta)\}.$$ 
By \cite[Proposition 3.6]{gimpgofflouc}, the compact manifold $S^2$ with its geodesic distance has property (MR), a property inherited by $X_\theta$. We compute that for $\theta\in (0,\pi)$, 
$$ \mathcal{M}_{X_\theta}(R)=(1-\cos(\theta))R^2+\sin(\theta)R+1+\frac{1}{8}\cos(\theta)+O(R^{-1}),$$
using that the sphere $S^2$ has scalar curvature $2$, the mean curvature of $\partial X_\theta$ is $H=-\mathrm{cotan}(\theta)$ and $\mathrm{vol}_2(X_\theta)=2\pi(1-\cos(\theta))$. 
In the limit cases $\theta\in \{0,\pi\}$, $\mathcal{M}_{X_0}(R)=\mathcal{M}_{\rm pt}(R)=1$ and $\mathcal{M}_{X_\pi}(R)=\mathcal{M}_{S^2}(R)=\frac{2R^2+2}{1-\mathrm{e}^{-\pi R}}\sim 2R^2+2$ (see Example \ref{shellexample}). Therefore, the coefficients $c_j$ jump in the limit cases analogous to Example \ref{shellexample}. 
\end{example}

\begin{example}[Hyperbolic space]
\label{hyperexample}
Consider the real hyperbolic $n$-space $M = \mathbb{H}^n$  with its Riemannian distance function $\rd$. Then $\rd^2$ is smooth, so that $M$ and any compact, smooth domain $X \subset M$ have property (MR). 

In dimension $n=3$ let $X_r \subset M$ be a ball of radius $r>0$, as depicted in Figure \ref{plot}(d) using the hyperboloid model of $M$ (for $n=2$). Theorem \ref{introthm} shows 
$$\mathcal{M}_{X_r}(R)=\frac{\sinh(2r)-2r}{8} R^3+\frac{\cosh(2r)-1}{2}R^2+(H_0(r)-1)(\cosh(2r)-1)R+O(1).$$
Indeed, the volume of $X_r$ is $\pi(\sinh(2r)-2r)$, the area of $\partial X_r$ is $2\pi(\cosh(2r)-1)$, $s=-6$ in real hyperbolic $3$-space and $H=H_0(r)$ is the mean curvature of the boundary of the ball, which only depends on $r$ by rotational invariance. 
\end{example}

\begin{example}[Torus]
\label{partoftorii}
For $k,l \in \mathbb{N}$ consider the manifold $M_1=S^k\times \R^l$ and the closely related $M_2=S^k\times \left(\R/2\mathbb{Z}\right)^l$, equipped with their natural geodesic distances. By an extension of the argument for $S^k$ in \cite[Proposition 3.6]{gimpgofflouc}, $M_1$ and any compact, smooth domain $X \subset M_1$ have property (MR). Property (MR) is not satisfied for $M_2$ if $k,l>0$, following \cite[Proposition 3.17]{gimpgofflouc}. Nevertheless, a domain $X \subset M_2$ which is small enough to be contained in the image of $S^k\times [0,1)^l\to M_2$ is isometric to a domain in $M_1$ and therefore has property (MR).

We may therefore use Theorem \ref{introthm} to compute the magnitude function of the armband domain depicted in Figure \ref{plot}(e), $X_\epsilon$ defined as the image of $S^1\times [0,\epsilon]\to M_2$, for $\epsilon\in (0,1)$. We obtain
$$\mathcal{M}_{X_\epsilon}(R)=\epsilon R^2+2R+O(R^{-1}),$$
since $s=0$ and $\int_{\partial X_\epsilon}H\mathrm{d}S=0$ by sign symmetry in the mean curvature. For $\epsilon \geq 1$ property (MR) is not satisfied, and Theorem \ref{introthm} does not apply.  The magnitude function of $\widetilde{X}_\epsilon = S^1\times [0,\epsilon] \subset M_1$ is given by $\mathcal{M}_{\widetilde{X}_\epsilon}(R)=\epsilon R^2+2R+O(R^{-1})$ for all $\epsilon >0$.
\end{example}

Let us now review the ideas and techniques behind the main results, including Theorem \ref{introthm}. As mentioned at the beginning of this section, at the heart of this paper is the analysis of the integral operator  $\mathcal{Z}_X(R)$ from \eqref{magequaq} using pseudodifferential methods, when $X$ is a manifold with boundary.  

Section \ref{sec:meckes} recalls the abstract function space framework introduced by Meckes \cite{meckes} for the magnitude operator $\mathcal{Z}_X(R)$ and connects it to concrete Sobolev spaces of distributions in $X$ and the analytic techniques developed in \cite{gimpgofflouc}. In particular, the approach in this article and in \cite{gimpgofflouc} therefore, indeed, computes the magnitude function.  The key result is Theorem \ref{symbcorboundary}: For $\mu = \frac{n+1}{2}$, $\mathrm{Re}(R)$ sufficiently large and $\mathrm{arg}(R)<\frac{\pi}{n+1}$, the operator $\mathcal{Z}_X(R)$ defines a holomorphic family of isomorphisms between the Sobolev spaces $\dot{H}^{-\mu}(X)$ and $\overline{H}^{\mu}(X)$ of supported, respectively extensible distributions. Results on the meromorphic continuation of $\mathcal{M}_X$ to the complex plane follow from the proof, see Corollary \ref{corofofrdlforx}. The abstract framework persists even when the boundary $\partial X$ is merely $C^0$ (i.e.~locally the graph of a continuous function). 

The idea in the proof is to replace $\mathcal{Z}_X$ with a localization 
$$Q_X(R)u(x):=\frac{1}{R} \int_X \chi(x,y)\e^{-R\rd(x,y)}u(y)\rd y, \quad R\in \C\setminus \{0\}.$$
Here, we fix a function $\chi\in C^\infty(X\times X)$ such that $\chi=1$ on a neighborhood of the diagonal $x=y$ and $\rd^2$ is smooth on the support of $\chi$. We explicitly compute in Theorem \ref{conomkmlog} that the operator $Q_X$ is a lower order perturbation of a fractional Laplacian with parameter $(R^2+\Delta)^{-\mu}$. Under the above stated conditions on $R$, $Q_X$ is invertible as a sufficiently small perturbation of the invertible fractional Laplacian. 

The main results of this article follow from the explicit analysis of the pseudodifferential operator $Q_X$. 
We assume that $\rd^2(x,y)$ is a divergence in the sense of \cite[Section 1.2]{divergenceref}, i.e.~$\rd^2(x,y)$ is a smooth function in a small neighborhood of the diagonal $x=y$, and in local coordinates it has a Taylor expansion (for any $N>0$)
\begin{equation}
\label{talaldadladldaladlda}
\rd(x,x-y)^2=H_{\rd^2}(x,x-y)+\sum_{j=3}^N C^{j}(x;x-y)+O(|x-y|^{N+1}),
\end{equation}
where $H_{\rd^2}(x,\cdot)$ is a Riemannian metric on $X$. In local coordinates $C^{j}$ is a symmetric $j$-form in $x-y$. The Taylor expansion \eqref{talaldadladldaladlda} translates into an expansion of the symbol of the pseudodifferential operator $Q_X$, see Theorem \ref{conomkmlog}, where terms for larger $j$ contribute to terms of the symbol of lower order jointly in $\xi$ and $R$. 

Section \ref{Mclosed} discusses the special case when $X=M$ is a closed manifold. In this case the inverse $Q_{X}(R)^{-1}$ is a pseudodifferential operator, whenever it exists, and the full symbol of $Q_X(R)^{-1}$ can be explicitly computed by the iterative scheme described in Proposition \ref{alksdnaksndasodn}. The operator $Q_X$ is generally better behaved than $\mathcal{Z}_X$. The off-diagonal singularities of $\rd$ may create problems when considering $\mathcal{Z}_X$ as a map between Sobolev spaces. Even for a Riemannian manifold these relate to difficult geometric questions about the structure of the cut-locus, cf. the discussion in \cite[Section 3]{gimpgofflouc}. Under suitable assumptions on $(X,\rd)$, such as (MR) and (SMR), properties of $Q_X$ are inherited by $\mathcal{Z}_X$, see Corollary \ref{penfpefnaon}. The asymptotic expansion of the magnitude function 
$$\mathcal{M}_X(R)=R^{-1}\langle \mathcal{Z}_X(R)^{-1}1,1\rangle_X \sim R^{-1}\langle Q_X(R)^{-1}1,1\rangle_X,$$ 
then follows with expansion coefficients $c_j$ computed from the symbol of $Q_X^{-1}$, and therefore from $C^{j}$. See Theorem \ref{magcompsclosedeldforz} and \ref{compactthm}.

The reader can find a discussion of the properties (MR) and (SMR) in Remark \ref{mrandsmr} and details in \cite[Section 3]{gimpgofflouc}. They are satisfied, for example, if $\rd^2$ is smooth on all of $X\times X$, such as for a domain or a submanifold in $\R^n$ with the induced metric.

The analysis in Section \ref{Mclosed} for a closed manifold illustrates the general approach taken in this article. For a manifold with boundary the inverse $Q_X(R)^{-1}$ decomposes into the previously studied interior part and a new boundary contribution, as described in Theorem \ref{reoslsdlnadkn}. Section \ref{sec:boundary} discusses the boundary contribution using methods for pseudodifferential boundary value problems, here of negative order $-\frac{n+1}{2}$. While such problems have a long history \cite{eskinbook}, the magnitude problem connects to recent developments for boundary problems for the fractional Laplacians \cite{g3, grubb1}. The resulting boundary contribution to the expansion coefficients $c_j$ involves the coefficients $C^j$ and the geometry of $\partial X$ inside $X$. Sections \ref{Mclosed} and \ref{sec:boundary} rely on and motivate the purely analytic results for $\mathcal{Z}_X$ in \cite{gimpgofflouc}, which we here exploit for geometric applications.

The methods presented in this article are amenable to a black-box computer implementation, discussed in Appendix \ref{appendixA}. Appendix \ref{appendixB} lists pseudocode for Euclidean domains and for submani\-folds of Euclidean space, which are the basis for the computational results in this article.\\

Beyond these main results of the article, the analysis of the boundary contributions in Subsection \ref{boundarlocalsos} sheds light 
 on the boundary behavior of the weight distribution recently considered for applications in data science. More precisely, in Subsection \ref{boundarlocalsos}, we use Lemma \ref{lasdnaodjna} to obtain a weak form of the structural properties conjectured in \cite{bunch2}. In Subsection \ref{tarolmlnad}, we connect to results of Meckes \cite{meckes20} and address the Taylor expansion of the magnitude function at $R=0$.

\section{Geometry of magnitude function and magnitude operator}
\label{sec:geometry}

The structural properties and geometric formulas for the expansion coefficients of the magnitude function as $R\to \infty$ shed light on the geometric content of magnitude. We here use Theorem \ref{introthm}, and more generally results from Subsection \ref{asymtppomomad} and \ref{subseconasus} below, as a tool to generate geometric consequences on magnitude.

\subsection{Relation to Euler characteristic} 
\label{eulerrel} 
Magnitude was originally motivated as a generalization of the Euler characteristic to finite (enriched) categories, and Leinster and Willerton conjectured that the expansion coefficient $c_n$ equals the Euler characteristic for a convex body in $\mathbb{R}^n$. Therefore, the relation between these two quantities is of interest.

\begin{prop}
\label{apfknaodfnaspodn}
Let $M$ be a compact Riemannian surface with geodesic distance $\rd$ and Riemannian metric $g$. Assume that $M$ admits a family $(u_R)_{R>R_0}\subseteq \mathcal{D}'(M)$ of distributional solutions to 
$$\int_M \e^{-R\rd(x,y)}u_R(y)\rd y=1, \quad R>R_0.$$
Assume that for any $\chi,\chi'\in C^\infty(M)$ with disjoint supports it holds that 
$$\int_M \chi(x)\e^{-R\rd(x,y)}u_R(y)\chi'(y)\rd y=O(R^{-3}).$$ 
Then the magnitude function $\mathcal{M}_M(R)$ is defined for $R>R_0$ and there is an asymptotic expansion 
$$\mathcal{M}_M(R)=\frac{\mathrm{vol}_2(M,g)}{2\pi}R^2+\frac{\chi(M)}{4}+O(R^{-1}),$$
where $\chi(M)$ denotes the Euler characteristic of $M$.
\end{prop}

We note that if $M$ is a homogeneous surface of compact type, then it satisfies the assumptions of the proposition as we can construct $(u_R)_{R>R_0}$ from an averaging procedure. See more in \cite{will}. For further discussion concerning this assumption, see Remark  \ref{compactthmrmk}. 

\begin{proof}
The assumptions on the existence of the distributional solution and Theorem \ref{compactthm}, a variant of Theorem \ref{introthm} stated in the introduction, implies that $\mathcal{M}_M(R)=\int_M u_R(y)\rd y$ is defined for $R>R_0$. Theorem \ref{compactthm} in fact gives us $\mathcal{M}_M(R)\sim \frac{1}{2\pi}\sum_{k=0}^2 c_k(M)R^{2-k}+O(R^{-1})$ for $c_2(M) = \frac{1}{2} \int_M \mathrm{s}$, where $s$ is the scalar curvature. The assertion then follows from the Gauss-Bonnet theorem for surfaces, $\int_M \mathrm{s} = \pi \chi(M)$.
\end{proof}

A corresponding connection between the constant term in $\mathcal{M}_X$ and the Euler characteristic fails for domains $X \subset \mathbb{R}^3$. With $H$ the mean curvature of $\partial X$, $c_3(X) = \delta \int_{\partial X} H^2$  is a multiple of the Willmore energy by \cite{gimpgoff4th}. The Willmore energy, however, can be arbitrarily large on surfaces of genus $0$ and is not determined by the Euler characteristic.

\subsection{Inclusion-exclusion principles}
\label{inadiandin}
As first shown in \cite{gimpgoff} for smooth, compact domains in a Euclidean space of odd dimension, one of the fundamental properties of the Euler characteristic still holds in an asymptotic form: the inclusion-exclusion principle. The results in this article imply an inclusion-exclusion principle for smooth, compact domains in manifolds of any dimension, under the condition (MR) from Remark \ref{MRremark}.

\begin{prop}\label{inex}
Let $M$ be a manifold with a distance function which satisfies $(MR)$. Let $X,Y \subset M$ be smooth, compact domains such that $X\cup Y$ and $X \cap Y$ are smooth. Then
$$\mathcal{M}_{X\cup Y}(R)=\mathcal{M}_X(R)+\mathcal{M}_Y(R)-\mathcal{M}_{X\cap Y}(R)+O(R^{-\infty}).$$
\end{prop}

\begin{proof}
The assertion follows from the local formulas for the expansion coefficients $c_j$ from Theorem \ref{technicalthmsec6} and the fact that property $(MR)$ is inherited by smooth, compact domains.
\end{proof}

\subsection{Can you ``magnitude the shape of a drum''}
\label{specgeo1}

In spectral geometry, analogues of the expansion coefficients $c_j$  have proven fruitful to find relationships between geometry 
and the eigenvalues of the Laplace-Beltrami operator, see \cite{notices} for a recent overview. The guiding question by M. Kac, ``Can you hear the shape of a drum?'' \cite{kac} has an analogue for the magnitude function: Does the magnitude function $\mathcal{M}_X$ determine a compact domain $X$ up to isometry? 

The answer is positive for the magnitude function of a ball: 
\begin{prop}\label{balldetermined}
Let $B \subset \mathbb{R}^n$ a ball. If $X \subset \mathbb{R}^n$ is a smooth, compact domain with $\mathcal{M}_X = \mathcal{M}_B +o(R^{n-1})$ for $R \to \infty$, then $X$ is isometric to $B$.
\end{prop}

\begin{proof}
By assumption 
\begin{align*}
n!\omega_n \mathcal{M}_X(R) &= \mathrm{vol}_n(X) R^{n} + {\textstyle \frac{n+1}{2}} \textnormal{vol}_{n-1}(\partial X) R^{n-1} + o(R^{n-1})  \\&= \mathrm{vol}_n(B) R^{n} +{\textstyle \frac{n+1}{2}} \textnormal{vol}_{n-1}(\partial B) R^{n-1} + o(R^{n-1}) = n!\omega_n\mathcal{M}_B(R) + o(R^{n-1}).
\end{align*}
Therefore $\mathrm{vol}_n(X) = \mathrm{vol}_n(B)$ and $\mathrm{vol}_{n-1}(\partial X) = \mathrm{vol}_{n-1}(\partial B)$. Recalling the isoperimetric inequality, $n \omega_n^{1/n} \mathrm{vol}_{n}(X)^{(n-1)/n} \leq \mathrm{vol}_{n-1}(\partial X)$, with equality if and only if $X$ is isometric to a ball, it follows that $X$ is isometric to $B$.
\end{proof}

\begin{remark}
For nonconvex domains, however, a counterexample by Meckes \cite{blog4} shows that the magnitude function $\mathcal{M}_X$ does not determine a compact domain $X$ up to isometry. Let $n$ be odd. Consider balls $B_1, B_2 \subset \mathbb{R}^n$ of the same diameter which are contained in the interior of a large ball $B$. Then $\mathcal{M}_{B\setminus B_1^\circ} = \mathcal{M}_{B\setminus B_2^\circ}$, but generically $B\setminus B_1^\circ$ and $B\setminus B_2^\circ$ are not isometric. 
\end{remark}

\subsection{Constant mean curvature} 
\label{specgeo2}

The proof of Proposition \ref{balldetermined} indicates the opportunities for studying the relationship between the geometry of a domain $X$ and its magnitude function $\mathcal{M}_X$ using techniques from spectral geometry, based on the expansion of $\mathcal{M}_X$ for $R \to \infty$. Using such techniques we obtain an analogue for magnitude of a theorem by Berger \cite{berger} that for a closed Riemannian surface having constant sectional curvature is determined by the eigenvalues of the Laplace-Beltrami operator.   

\begin{prop}\label{cmcberger}
Let $X,Y \subset \mathbb{R}^n$ be smooth, compact domains and $n$ odd or $n=2$, and assume that $\partial X$ has constant mean curvature $H(\partial X) =H$. If $\mathcal{M}_X = \mathcal{M}_Y +o(R^{n-3})$ for $R \to \infty$, then also $\partial Y$ has constant mean curvature $H(\partial Y) = H$.
\end{prop}

This result generalizes Proposition \ref{balldetermined}. 

\begin{proof}
Note that the polynomial $p_Y(z) = \int_{\partial Y}(z-H(\partial Y))^2$ has a real zero if and only if  the mean curvature of $\partial Y$ is constant. Since $c_j(X)\propto \int_{\partial X}H^{j-1}\rd S$ for $j=1,2,3$ if $n=2$ (see Theorem \ref{diskthm} and Remark \ref{computersays}) or $n$ is odd (see \cite{gimpgoff4th}), there exist constants $\alpha_n,\beta_n, \gamma_n$ depending only on the dimension $n$ such that 
$$p_Y(z)= \alpha_n c_1(Y) z^2 + \beta_n c_2(Y) z + \gamma_n c_3(Y) \ .$$
Because $c_j(Y) = c_j(X)$ for $j=1,2,3$, $p_Y(z) = \int_{\partial X}(z-H(\partial X))^2$. With $H(\partial X) = H$ constant, $p_Y$ vanishes in $H$, and therefore $H(\partial Y)$ is constant and equals $H$.
\end{proof}

\subsection{Relation to residues of manifolds}\label{residuessec}

For a compact metric space $(X,\rd)$, the geometric relevance of the magnitude operator $\mathcal{Z}_X(R)$ is not restricted to the magnitude function $\mathcal{M}_X(R) =R \langle \mathcal{Z}_X(R)^{-1}1,1\rangle$. Closely related quantities of interest are the meromorphic energy function 
$$B_X(z) := \mathrm{F.P.}|_{s=z}\int_{X\times X} \rd(x,y)^s\ \rd x\ \rd y\ ,$$ 
and the residue 
$$R_X(z):=\mathrm{res}_{s=z}\int_{X\times X} \rd(x,y)^s\ \rd x\ \rd y\ .$$  

The energy $B_M(z)$ was first introduced by Brylinski \cite{B} for knots in $\mathbb{R}^3$ and by Fuller and Venmuri \cite{FV} for closed submanifolds of $\mathbb{R}^n$. $B_M(z)$ and $R_M(z)$ may be expressed in terms of $\mathcal{Z}_X$ by means of the formula 
$$\rd^{-s} = \frac{1}{\Gamma(s)}\int_0^\infty  R^{s-1} e^{-R \rd}\ \rd R,$$ 
valid for $\mathrm{Re}(s)>0$.
Then 
\begin{align*}
\int_{X\times X} \rd(x,y)^{-s}\ \rd x\ \rd y &= 
\frac{1}{\Gamma(s)}\int_0^\infty  R^{s-1} \int_{X\times X} e^{-R \rd(x,y)}\ \rd x\ \rd y\ \rd R\\ 
&= \frac{1}{\Gamma(s)}\int_0^\infty  R^{s-1} \langle \mathcal{Z}_X(R) 1,1\rangle\ \rd R\ . 
\end{align*}
To understand the relation between the residues of 
$$f(s) = \int_{X\times X} \rd(x,y)^{-s}\ \rd x\ \rd y$$ 
and the expectation value of the magnitude operator, 
$$e(t) = \langle \mathcal{Z}_X(t)1,1\rangle =\int_{X\times X}e^{-t\rd(x,y)}\rd x\ \rd y\ ,$$
we divide the integral on the right hand side and set $\lambda = R^{-1}$:
\begin{align*}
f(s) &= \frac{1}{\Gamma(s)}\int_0^1  R^{s-1} e(R)\ \rd R + \frac{1}{\Gamma(s)}\int_1^\infty  R^{s-1} e(R)\ \rd R\\ & = 
\frac{1}{\Gamma(s)}\int_0^1  R^{s-1} e(R)\ \rd R + \frac{1}{\Gamma(s)}\int_0^1  \lambda^{-(s+1)} e(\lambda^{-1})\ \rd \lambda\ .
\end{align*}
The proof of the following proposition is analogous to \cite[Proposition 5.1]{gs}.

\begin{prop}
\label{relarlandlna}
Assume that $e(t)$ and $e(t^{-1})$ are holomorphic in $V_{\theta_0}$ (for some $\theta_0 \in (0,\frac{\pi}{2})$), $$V_{\theta_0} = \{t = re^{i\theta} :\ 2>r>0, |\theta|<\theta_0\}\ ,$$ and $e(t) = O(|t|^a)$, $e(t^{-1}) = O(|t|^{b})$ for $t\to 0$ in $V_\delta$, any $\delta<\theta_0$, for some $a,b\in \mathbb{R}$. 
Consider the function
$$f(s) = \int_0^\infty t^{s-1} e(t) \ \rd t,$$ 
holomorphic for $\mathrm{Re}(s)>-a$. Then the following properties are equivalent:
\begin{enumerate}
\item[(a)] $e(t)$ and $e(t^{-1})$ have asymptotic expansions for $t \to 0$, 
\begin{align*}
e(t) &\sim \sum_{j=0}^\infty a_{j} t^{\beta_j}, \ \ \mbox{where}\ \ \beta_j \to +\infty,\\
e(t^{-1}) &\sim \sum_{j=0}^\infty  A_{j} t^{\gamma_j} , \ \ \mbox{where}\ \ \gamma_j \to +\infty, 
\end{align*}
uniformly for $t \in V_\delta$, for each $\delta<\theta_0$.\\
\item[(b)] $f(s)$ is meromorphic on $\mathbb{C}$ with the singularity structure 
$$\Gamma(s) f(s) \sim \sum_{j=0}^\infty \frac{a_{j}}{\beta_j+s} + \sum_{j=0}^\infty \frac{A_{j}}{\gamma_j-s},\ $$
and for each real $C_1, C_2$ and each $\delta<\theta_0$,
$$|f(s)| \leq C(C_1,C_2,\delta)e^{-\delta |\mathrm{Im}(s)|}, \ \ |\mathrm{Im}(s)|\geq 1, \ C_1\leq |\mathrm{Re}(s)| \leq C_2\ .$$
\end{enumerate}
\end{prop}

The assumptions on $e(t)=\langle \mathcal{Z}_X(t)1,1\rangle$ from Proposition \ref{relarlandlna} are trivially satisfied; indeed $e(t)=O(t^{-\infty})$ as $t\to \infty$ and $e(t)$ is real analytic near $0$. There is therefore, by Proposition \ref{relarlandlna}, a close relation between the expectation value of the magnitude operator $\langle \mathcal{Z}_X(t)1,1\rangle$ and the residue for an arbitrary metric space $(X,\rd)$.

The geometric content discovered for residues of knots, hypersurfaces and domains therefore translates into geometric content for the corresponding magnitude operator:

\begin{prop}
\label{prop_residues_closed_12}
\cite{B,FV,OS2} Let $M$ be a compact hypersurface in $\mathbb{R}^{m+1}$ with principal curvatures $\kappa_i$. Define  ${\Vert h\Vert}^2=\sum_i{\kappa_i}^2$ and $o_{k}$ the surface area of the unit $k$-sphere. 
\begin{enumerate}
\item The first residue is given by $\displaystyle R_M(-m)=\displaystyle o_{m-1}\,\mbox{\rm Vol}\,(M).$
\item The second residue is given by 
$$R_{M}(-m-2)=\frac{o_{m-1}}{8m}\int_M\left(2{\Vert h\Vert}^2-{m^2\vert H\vert}^2\right)\ \rd x .$$
\end{enumerate}
\end{prop}

\begin{prop}
\label{prop_residues_closed_123} 
\cite{OS2} Let $\Omega \subset \mathbb{R}^n$ be a compact domain. 
\begin{align*}
R_{\Omega}(-n)&=\displaystyle {o_{n-1}}\text{\rm Vol}\,(\Omega),\\
R_{\Omega}(-n-1)&=-\frac{o_{n-2}}{n-1}\,\textrm{\rm Vol}(\partial \Omega), \\
R_\Omega(-n-3)&=\displaystyle \frac{o_{n-2}}{24(n^2-1)}\int_{\partial \Omega} \left(3n^2 H^2-2\mathrm{s}\right)\ \rd x ,   
\end{align*}
\end{prop}

Also an inclusion-exclusion principle was derived for residues \cite{ohara}.

\subsection{Computations for cylinders}

In the special case of a cylinder of the form $X_T=M\times [0,T]$, for $T>0$, the magnitude function of $X_T$ relates to geometric invariants of $M$. Let us compute an instance of this. 

\begin{prop}
\label{cylcompmpmad}
Assume that $M\subseteq \R^N$ is an $n-1$-dimensional compact submanifold. For $T>0$, consider the compact submanifold with boundary 
$$X_T:=M\times [0,T]\subseteq \R^{N+1}.$$
Equipping $M$ and $X_T$ with the subspace distance, we have that 
\begin{align*}
\mathcal{M}_{X_T}(R)\sim \ &\frac{\mathrm{vol}_{n-1}(M)}{n!\omega_n} TR^n+\frac{(n+1)\mathrm{vol}_{n-1}(M)}{2\cdot n!\omega_n} R^{n-1}+\\
&+\frac{n+1}{6\cdot n!\omega_n}\int_M s_{\rd^2,M}\rd x \ TR^{n-2}+\sum_{k=4}^\infty \rho_k(M)TR^{n-k},
\end{align*}
where $\rho_k(M)=\int_M a_{k,0}(x,1)\rd x$ is computed as an integral over $M$ from the local densities $a_{k,0}$ on $X_T$ of Theorem \ref{technicalthmsec6} and $\rho_k(M)=0$ if $k$ is odd.
\end{prop}

\begin{proof}
On $X_T$, the last coordinate in $\R^{N+1}$ provides a global coordinate $x_n$ for the direction transversal to the boundary, and it is a normal for the Riemannian metric associated with the distance function. Therefore, there are no boundary contributions to the magnitude asymptotics par $c_1(X_T)$ which is proportional to the volume of the boundary. The format for $c_0(X_T)$ and $c_2(X_T)$ follows from Theorem \ref{technicalthmsec6} since all structures are constant in the $x_n$-direction. Therefore, the description $c_k(X_T)=\int_{X_T} a_{k,0}(x,1)\rd x=T\int_M a_{k,0}(x,1)\rd x=T\rho_k(M)$, $k>2$, follows from Theorem \ref{technicalthmsec6}.
\end{proof}

\subsection{Asymptotically polynomial behavior in dimension $3$}

In dimension $3$, the magnitude function of the unit ball is a polynomial by \cite{barcarbs}. To the knowledge of the authors, for Euclidean domains, this can only happen in dimensions $1$ and $3$. Based on the computational evidence from Theorem \ref{Rntheorem} and Remark \ref{computersays} for the statement that $c_j(X)\propto \int_{\partial X}H^{j-1}\rd S$, for $j>0$ and $X\subseteq \R^n$ a domain, we make the following observation. 

\begin{prop}
Let $N>3$. The following two statements are equivalent:
\begin{enumerate}
\item There are coefficients $\gamma_1,\gamma_2,\ldots,\gamma_N$ such that for any domain $X\subseteq \R^3$ with smooth boundary we have that
$$c_j(X)=\gamma_j \int_{\partial X}H^{j-1}\rd S, \quad j=1,\ldots, N.$$
\item There is a coefficient $\gamma_3$ such that for any domain $X\subseteq \R^3$ with smooth boundary we have that
$$\mathcal{M}_X(R)=\frac{\mathrm{vol}_3(X)}{8\pi}R^3+\frac{\mathrm{vol}_2(\partial X)}{4\pi}R^2+\frac{1}{2\pi}\int_{\partial X}H\rd S R+\gamma_3\int_{\partial X}H^2\rd S+O(R^{2-N}).$$
\end{enumerate}
\end{prop}

\begin{proof}
It is clear from Theorem \ref{Rntheorem}  that 2) implies 1), with $\gamma_j=0$ for $j=4,\ldots, N$. If 1) holds, then we have for any domain $X\subseteq \R^3$ with smooth boundary that 
$$\mathcal{M}_X(R)=\frac{\mathrm{vol}_3(X)}{8\pi}R^3+\frac{1}{8\pi}\sum_{j=1}^N \gamma_j\int_{\partial X}H^{j-1}\rd S R^{3-j}+O(R^{2-N}).$$
By  \cite[Theorem 2]{barcarbs}, $\mathcal{M}_X(R)$ is a polynomial for the unit ball. In particular, we must have that $\gamma_j=0$ for $j=4,\ldots, N$. The precise form of $\gamma_1$ and $\gamma_2$ can be found in \cite{gimpgoff} or Theorem \ref{Rntheorem} below.
\end{proof}

\section{Weight distributions and magnitude}
\label{sec:meckes}
The work of Meckes \cite{meckes} relates the magnitude to a capacity-like definition for a reproducing kernel Hilbert space defined from $(X,\rd)$. We here recall this approach in a form which connects it to the analytic techniques developed in \cite{gimpgofflouc}. This will show that the approach of  \cite{gimpgofflouc} indeed computes the magnitude function. 

Let $(M,\rd)$ be a compact metric space. Consider the vector space $\mathrm{FM}(M)$ of finitely supported complex measures on $M$. For $R\in \C$, we define the form $\langle\cdot,\cdot\rangle_{\mathcal{W}_R}$ on $\mathrm{FM}(M)$ by
\begin{equation}
\label{wrform}
\langle\mu,\nu\rangle_{\mathcal{W}_R}:=\int_{M\times M}\e^{-R\rd(x,y)}\rd\bar{\mu}(x)\rd\nu(y).
\end{equation}
For real $R$, $\langle\mu,\nu\rangle_{\mathcal{W}_R}$ is a sesquilinear form on $\mathrm{FM}(M)$. The following result follows from the definition of a positive definite metric space. 

\begin{prop}
\label{pospos}
Let $R>0$. The form $\langle\cdot,\cdot\rangle_{\mathcal{W}_R}$ is positive definite on $\mathrm{FM}(M)$ if and only if $(M,R\cdot \rd)$ is positive definite. 
\end{prop}

For $R>0$ such that $(M,R\cdot \rd)$ is positive definite, we define $\mathcal{W}_R(M)$ as the completion of $\mathrm{FM}(M)$ in the inner product $\langle\cdot,\cdot\rangle_{\mathcal{W}_R}$. Note that for any compact $X\subseteq M$ there is an isometric inclusion of Hilbert spaces 
$$\mathcal{W}_R(X)\subseteq \mathcal{W}_R(M).$$
Let $C^{1/2}(M,\rd)$ denote the Banach space of functions on $M$ which are H\"older continuous of exponent $1/2$.

\begin{prop}[Proposition 3.2 of \cite{meckes}]
For $R>0$ such that $(M,R\cdot \rd)$ is positive definite, the operator 
$$\mathcal{Z}(R)\mu(x):=\frac{1}{R}\int_M\e^{-R\rd(x,y)}\rd\mu(x),\quad \mu\in \mathrm{FM}(M)$$
extends to a continuous mapping 
$$\mathcal{Z}(R): \mathcal{W}_R(M)\to C^{1/2}(M,\rd).$$
Moreover, for any $\mu \in \mathrm{FM}(M)$, $\nu\in  \mathcal{W}_R(M)$, 
\begin{equation}
\label{dualityforz}
\langle \mu,\nu\rangle_{\mathcal{W}_R}=R\int_M [\mathcal{Z}(R)\nu](x)\rd \bar{\mu}(x).
\end{equation}
\end{prop}

We define $\mathcal{H}_R(M):=\mathcal{Z}(R)\mathcal{W}_R\subseteq C^{1/2}(M,\rd)$. The vector space $\mathcal{H}_R(M)$ becomes a Hilbert space by declaring  $\mathcal{Z}(R):\mathcal{W}_R(M)\to \mathcal{H}_R(M)$ to be a unitary isomorphism. Equation \eqref{dualityforz} shows that there is a canonical identification $\mathcal{H}_R(M)=\mathcal{W}_R(M)^*$ given by a pairing 
\begin{equation}
\label{dualityforzpairing}
\langle f,\mu\rangle_{L^2,R}:=\langle \mathcal{Z}(R)^{-1}f,\mu\rangle_{\mathcal{W}_R}=\langle f,\mathcal{Z}(R)\mu\rangle_{\mathcal{H}_R},\quad f\in \mathcal{H}_R(M), \; u\in \mathcal{W}_R(M).
\end{equation}
We call this pairing the $L^2$-pairing because, when $M$ is a manifold, the $L^2$-pairing is independent of $R$ and coincides with the ordinary $L^2$-pairing. By duality, for any compact $X\subseteq M$ the inclusion $\mathcal{W}_R(X)\subseteq \mathcal{W}_R(M)$ induces a restriction mapping 
$$\mathcal{H}_R(M)\to \mathcal{H}_R(X).$$
Since $\mathcal{W}_R(X)\subseteq W_R(M)$ as Hilbert spaces, the restriction mapping $\mathcal{H}_R(M)\to \mathcal{H}_R(X)$ is a co-isometry. By construction, we have a commuting diagram
$$\begin{CD}
\mathcal{W}_R(X) @>\mathcal{Z}(R) >> \mathcal{H}_R(X) \\
@VVV @AAA \\ 
\mathcal{W}_R(M) @>\mathcal{Z}(R) >> \mathcal{H}_R(M)
\end{CD}$$

We note that since $\mathcal{H}_R\subseteq C^{1/2}(M,\rd)$ is a continuous inclusion, for any compact $X\subseteq M$ and $h\in \mathcal{H}_R$ the restriction $h|_X\in C^{1/2}(X,\rd)$ is well defined. 

\begin{thm}[Section 3 of \cite{meckes}]
\label{mexkesmsm}
Let $X\subseteq M$ be a compact subset and $R>0$ such that $(M,R\cdot \rd)$ is positive definite. If there is an element $u_R\in \mathcal{W}_R(X)$ such that $h_R:=R\mathcal{Z}(R)\mu_R\in \mathcal{H}_R(M)$ satisfies 
$$h_R|_X=1,$$
then 
$$\mathrm{mag}(X,R\cdot \rd)=\|u_R\|_{\mathcal{W}_R}^2.$$
\end{thm}

Since $\mathcal{W}_R(X)\subseteq \mathcal{W}_R(M)$, we trivially have that
\begin{equation}
\|u\|_{\mathcal{W}_R(M)}^2=\|u\|_{\mathcal{W}_R(X)}^2,\quad\forall u\in \mathcal{W}_R(X).
\end{equation}

We note the following corollary that will play an important role for applications to manifolds (that possibly have boundary).

\begin{cor}
\label{corofofrdl}
Assume $(M,\rd)$ is a compact metric space and that $R_0>0$ is such that $(M,R\cdot \rd)$ is positive definite for all $R>R_0$. Let $X\subseteq M$ be a compact subset and assume the following:
\begin{enumerate} 
\item There are Hilbert spaces $\mathcal{H}(M)$, $\mathcal{H}(X)$, $\mathcal{W}(M)$ and $\mathcal{W}(X)$ such that for all $R>R_0$, 
$$\mathcal{H}_R(M)=\mathcal{H}(M),\;\mathcal{H}_R(X)=\mathcal{H}(X), \; \mathcal{W}_R(M)=\mathcal{W}(M)\;\mbox{and}\; \mathcal{W}_R(X)=\mathcal{W}(X),$$
as vector spaces with equivalent norms.
\item The space $\mathcal{H}(X)\subseteq C^{1/2}(X,\rd)$ contains the constant function (the inclusion is induced from item (1)).
\item The $L^2$-pairings between $\mathcal{W}(X)$ and $\mathcal{H}(X)$, and between $\mathcal{W}(M)$ and $\mathcal{H}(M)$ induced from item (1) above are independent of $R$. 
\end{enumerate}
Writing $\mathcal{Z}_X(R)$ for the operator $\mathcal{W}(X)\to \mathcal{H}(X)$ induced from $\mathcal{Z}(R)$, we then have that  $\mathcal{Z}_X(R)$ is invertible for $R>R_0$ and the magnitude function is given by
\begin{equation}
\label{formulaformagoaod}
\mathcal{M}_X(R)=R^{-1}\langle 1,\mathcal{Z}_X(R)^{-1}1\rangle_{L^2},\quad\mbox{for $R>R_0$}.
\end{equation}
Moreover, if there is a connected domain $\R_{>R_0}\subseteq \Gamma\subseteq \C$ such that $\R_{>R_0}\ni R\mapsto \mathcal{Z}_X(R)\in \mathbb{B}(\mathcal{W}(X), \mathcal{H}(X))$ extends to a holomorphic Fredholm valued function $\Gamma\to  \mathrm{Fred}(\mathcal{W}(X), \mathcal{H}(X))$, then the magnitude function $\mathcal{M}_X$ extends to a meromorphic function on $\Gamma$.
\end{cor}

Item (3) listed in the assumptions of Corollary \ref{corofofrdl} is purely cosmetic and ensures that the $L^2$-pairing in Equation \eqref{formulaformagoaod} does not depend on $R$. For context, the reader should note that when $M$ is a manifold and $X\subseteq M$ is a smooth domain we can take $\mathcal{W}$ and $\mathcal{H}$ to be certain Sobolev spaces by Theorem \ref{symbcorboundary} below.

\begin{proof}
Since $(M,R\cdot \rd)$ is positive definite for $R>R_0$, the operator $\mathcal{Z}_X(R)$ is well defined and invertible using item (1). By item (2), the constant function $1$ on $X$ is an element of $\mathcal{H}(X)$ and we can define $u_R:=R^{-1}\mathcal{Z}_X(R)^{-1}1\in \mathcal{W}(X)$. Note that by item (1), $u_R\in \mathcal{W}_R(X)$ for all $R>R_0$. By Theorem \ref{mexkesmsm} and Equation \eqref{dualityforzpairing}, we have that 
$$\mathcal{M}_X(R)=\|u_R\|_{\mathcal{W}_R(M)}^2=R^{-1}\langle 1,u_R\rangle_{L^2,R}=R^{-1}\langle 1,\mathcal{Z}_X(R)^{-1}1\rangle_{L^2}$$
In the last equality we used item (3) to remove the dependence of $R$ in the pairing. 

The statement concerning the meromorphic extension of the magnitude function follows from the meromorphic Fredholm theorem which shows that $\R_{>R_0}\ni R\mapsto \mathcal{Z}_X(R)^{-1}\in \mathbb{B}(\mathcal{H}(X), \mathcal{W}(X))$ extends to a meromorphic function $\Gamma\to  \mathbb{B}(\mathcal{H}(X), \mathcal{W}(X))$ which extends the function $\mathcal{M}_X(R)=R^{-1}\langle 1,\mathcal{Z}_X(R)^{-1}1\rangle_{L^2}$ to a meromorphic function of $R\in \Gamma$.
\end{proof}

\begin{remark}
\label{mrandsmr}
In the work \cite{gimpgofflouc} the notions of distance functions having property (MR) and (SMR) were introduced. We refer the reader to  \cite[Section 3]{gimpgofflouc} for full details. The reader should keep in mind that (SMR)$\Rightarrow$ (MR) and that both (MR) and (SMR) are inherited by smooth subdomains. Domains and subspace distances in Euclidean space satisfy property (SMR). More generally, a compact subdomain or a subspace $X$ of a manifold $M$ satisfies (SMR) if (SMR) holds on $M$, e.g. if the restriction of $\rd^2$ to $X\times X$ is smooth. For the geodisic distance this occurs when the diameter of $X$ is smaller than the injectivity radius of $M$. The sphere $S^n$ with its geodesic distance has property (MR) but not (SMR) by \cite[Proposition 3.6]{gimpgofflouc}. Tori and real projective space with their geodesic distances fail to satisfy (MR) in dimension $>1$ by \cite[Proposition 3.17]{gimpgofflouc}. In the special cases we know that (MR) fails, the magnitude asymptotics can nevertheless be computed by the same formalism as in Theorem \ref{introthm}, because Theorem \ref{compactthm} below applies.
\end{remark}

To study the operator $\mathcal{Z}_X$ and the magnitude for manifolds, potentially with boundary, we introduce the following scales of Sobolev spaces. Let $X$ be a compact domain in a manifold $M$. For $s\in \R$, write
$$\dot{H}^s(X):=\{u\in H^s(M): \mathrm{supp}(u)\subseteq X\},\quad\mbox{and}\quad \overline{H}^s(X):=H^s(M)/\dot{H}^s(M\setminus X).$$
Here $H^s(M)$ is the Sobolev space of order $s$ on $M$. If $M$ is compact, $H^s(M)$ is unambiguously defined. If $M$ is non-compact, we can either define the Sobolev scale in terms of a Riemannian structure or by replacing the Sobolev spaces by local Sobolev spaces: the definitions of $\dot{H}^s(X)$, $\overline{H}^s(X)$ on the compact domain $X$ do not depend on this choice. We note that for $s=0$, $\dot{H}^0(X)=\overline{H}^0(X)=L^2(X)$. The $L^2$-pairing between $\dot{H}^s(X)$ and $\overline{H}^{-s}(X)$ is a perfect pairing inducing an isomorphism $\dot{H}^s(X)^*\cong \overline{H}^{-s}(X)$. We recall a relevant theorem from \cite{gimpgofflouc}. 

\begin{thm}
\label{symbcorboundary}
Let $X$ be an $n$-dimensional compact manifold with $C^0$-boundary and a distance function $\rd$ satisfying property (MR). Set $\mu:=(n+1)/2$. Then there exists an $R_0>0$ such that
$$\mathcal{Z}_X(R):\dot{H}^{-\mu}(X)\to \overline{H}^{\mu}(X),$$ 
is a well defined invertible operator for all $R\in \Gamma$ with $\arg(R)<\pi/(n+1)$ and $\mathrm{Re}(R)>R_0$. Moreover, $R_0$ can be chosen so that the following holds:
\begin{enumerate}
\item[a)] There is a $C>0$ such that 
$$C^{-1}\|f\|_{\dot{H}^{-\mu}_R(X)}^2\leq \mathrm{Re} \langle f,\mathcal{Z}_X(R) f\rangle_{L^2}\leq C\|f\|_{\dot{H}^{-\mu}_R(X)}^2,$$
for $R\in \{R\in \Gamma: \arg(R)<\pi/(n+1)\; \mbox{and}\; \mathrm{Re}(R)>R_0\}$ and $f\in \dot{H}^{-\mu}(X)$. 
\item[b)] For $R>R_0$, the sesquilinear form $R^{-1}\langle \cdot, \cdot \rangle_{\mathcal{W}_R}$ is uniformly equivalent to the inner product of $\dot{H}^{-\mu}_R(X)$. In particular, for $R>R_0$:
\begin{itemize}
\item The metric space $(X,R\cdot \rd)$ is positive definite.
\item We have the equalities $\mathcal{W}_R(X)=\dot{H}^{-\mu}_R(X)=\dot{H}^{-\mu}(X)$ and $\mathcal{H}_R(X)=\overline{H}^{\mu}_R(X)=\overline{H}^{\mu}(X)$ as vector spaces with equivalent norms.
\end{itemize}
\end{enumerate}
Moreover, if $\rd$ has property (SMR) on a sector $\Gamma$ the operator 
$\mathcal{Z}_X(R):\dot{H}^{-\mu}(X)\to \overline{H}^{\mu}(X)$ depends holomorphically on $R\in \Gamma$ and $\mathcal{Z}_X(R)^{-1}:\overline{H}^{\mu}(X)\to \dot{H}^{-\mu}(X)$ depends holomorphically on $R\in \{R\in \Gamma: \arg(R)<\pi/(n+1)\; \mbox{and}\; \mathrm{Re}(R)>R_0\}$ with a meromorphic extension to $R\in \Gamma$.
\end{thm}

\begin{remark}
As mentioned above, and proven in \cite[Proposition 3.4]{gimpgofflouc}, a distance function $\rd$ obtained by pulling back the Euclidean distance function along an embedding $X\hookrightarrow \R^N$, has property $(SMR)$ on $\C\setminus \{0\}$. In particular, for such distance functions the operator 
$$\mathcal{Z}_X(R)^{-1}:\overline{H}^{\mu}(X)\to \dot{H}^{-\mu}(X),$$ 
has a meromorphic extension to $R\in \C\setminus \{0\}$.
\end{remark}

For the full proof of Theorem \ref{symbcorboundary} we refer to \cite[Section 4]{gimpgofflouc}, but we sketch the idea here. The idea in the proof is to use that property $(SMR)$ allows us to replace $\mathcal{Z}_X$ with a localization $Q_X$ to an operator whose integral kernel is supported near the diagonal. The operator $Q_X$ is an elliptic pseudodifferential operator with parameter $R$ of order $-n-1$, we discuss $Q_X$ in further detail below in Theorem \ref{conomkmlog}. A computation of the principal symbol of $Q_X$ shows that $Q_X$ is a lower order perturbation of a fractional Laplacian with parameter $(R^2+\Delta)^{-(n+1)/2}$ -- the Laplacian comes from a Riemannian metric defined from the Taylor expansion of the distance function and the fractional power is defined relative to an extension to an ambient manifold. Therefore item b) follows from the G\aa rding inequality. Item a) follows from item b) using the meromorphic Fredholm theorem. Item c) is a direct consequence of item b).

\begin{cor}
\label{corofofrdlforx}
Let $X$ be an $n$-dimensional compact manifold with $C^0$-boundary and a distance function $\rd$ satisfying property (MR) on a sector $\Gamma$ with non-trivial intersection with $[0,\infty)$. Then the magnitude function is given by
\begin{equation}
\label{formulaformagoaod}
\mathcal{M}_X(R)=R^{-1}\langle 1,\mathcal{Z}_X(R)^{-1}1\rangle_{L^2},\quad\mbox{for $R>R_0$},
\end{equation}
where $\mathcal{Z}_X(R):\dot{H}^{-\mu}(X)\to \overline{H}^{\mu}(X)$. If moreover $\rd$ satisfies property (SMR) on $\Gamma$, the magnitude function $\mathcal{M}_X$ extends meromorphically to $\Gamma$ and is holomorphic in the sector $\{R\in \Gamma: \arg(R)<\pi/(n+1)\; \mbox{and}\; \mathrm{Re}(R)>R_0\}$ for some $R_0>0$.

In particular, if $X\subseteq \R^N$ is a compact submanifold with $C^0$-boundary (e.g.~a domain with $C^0$-boundary) with the subspace distance, then $\mathcal{M}_X$ extends meromorphically to $\C\setminus \{0\}$ and can be computed from \eqref{formulaformagoaod}.
\end{cor}

\section{The operator $\mathcal{Z}$ on a closed manifold}\label{Mclosed}

To better understand the operator $\mathcal{Z}$ we first consider the case of a compact manifold $M$. We here give an informal review of the technical considerations in the paper  by \cite{gimpgofflouc}. As discussed in the previous section, there are complications arising from the fact that distance function might be non-smooth away from the diagonal despite being quite regular at the diagonal.

\subsection{Localizing to the diagonal}
We decompose the operator $\mathcal{Z}$ in a part near the diagonal and an off-diagonal remainder:
$$\mathcal{Z}=Q+L,$$
where 
\begin{equation}
\label{arforgo}
Q(R)f:=\frac{1}{R}\int_M\chi(x,y)\e^{-R\rd(x,y)}f(y)\mathrm{d}y,
\end{equation}
and $\chi\in C^\infty(M\times M)$ satisfies $\chi=1$ near the diagonal $\mathrm{Diag}_M:=\{(x,x):x\in M\}\subseteq M\times M$. We first study the operator $Q$ and return below to study the remainder term $L$. In order to control $Q$, we impose the following condition on the distance function $\rd$.

\begin{deef}
\label{regularaadnai0dn}
Let $\rd$ be a distance function on a manifold $M$. We say that $\rd$ is \emph{regular at the diagonal} if its square $G:=\rd^2:M\times M\to [0,\infty)$ satisfies that there is a neighborhood $U$ of the diagonal $\mathrm{Diag}_M\subseteq M\times M$ such that:
\begin{itemize}
\item $G$ restricts to a smooth function on $U$;
\item $\rd G$ vanishes on $\mathrm{Diag}_M$; and,
\item for each $x\in \mathrm{Diag}_M$, the transversal Hessian $H_{G}$ of $G$ in $x$ is positively definite.
\end{itemize}
If $\rd$ is regular at the diagonal, we write $g_\rd$ for the Riemannian metric on $T^*M$ dual to the  transversal Hessian $H_{G}$.
\end{deef}

\begin{remark}
Examples of distance functions regular at the diagonal include subspace distances on submanifolds in Euclidean space (see \cite[Example 2.16]{gimpgofflouc}) or geodesic distance functions on Riemannian manifolds (see \cite[Example 2.17]{gimpgofflouc}). Then the Riemannian metric coincides with $g_\rd$.
\end{remark}

Consider a distance function $\rd$ regular at the diagonal. It follows from Taylor's theorem that for any $N\in \N$ in local coordinates on a neighborhood $U_0$ we can write 
\begin{equation}
\label{taylor}
\rd(x,y)^2=H_{\rd^2}(x,x-y)+\sum_{j=3}^NC^{(j)}(x,x-y)+r_N(x,x-y),
\end{equation}
where $r_N$ is a smooth function with $r_N(x,v)=O(|v|^{N+1})$ as $v\to 0$, $H_{\rd^2}$ is the transversal Hessian of $\rd^2$, and $C^{(j)}:U_0\to \mathrm{Sym}^j(T^*M|_{U_0})$ are the Taylor coefficients forming a locally defined symmetric $j$-form on $TM|_{U_0}$. 

In the symbol computation of $Q$, the Taylor coefficients are used as differential operators in the cotangent variable. For a $k\in \N_+$ and a multiindex $\gamma\in \N^k_{\geq 3}$, we can define a differential operator $C^{(\gamma)}(x,D_\xi)$ on $T^*M|_{U_0}$ defined by 
$$C^{(\gamma)}(x,-D_\xi):=\prod_{l=1}^kC^{(\gamma_l)}(x,-D_\xi).$$
Here $D_\xi=-i\frac{\partial}{\partial \xi}$. The order of $C^{(\gamma)}(x,-D_\xi)$ is $|\gamma|:=\sum_{l=1}^k\gamma_l$. For $j\in \N$, define the finite set
$$I_j:=\{\gamma\in \cup_{k=1}^\infty \N^k_{\geq 3}: |\gamma|=j+2k\}.$$
For $\gamma\in \N^k$, we set $\mathrm{rk}(\gamma):=k$. In other words, $\gamma\in \cup_k \N^k_{\geq 3}$ belongs to $I_j$ if and only if $j=|\gamma|-2\mathrm{rk}(\gamma)$. We remark that $|\gamma|\geq 3$ and $\mathrm{rk}(\gamma)>0$ is implicit for $\gamma \in I_j$ since $I_j\subseteq \cup_{k=1}^\infty \N^k_{\geq 3}$.

\begin{thm}
\label{conomkmlog}
Let $M$ be an $n$-dimensional manifold equipped with a distance function $\rd$ regular at the diagonal. Consider the operator $Q$ from Equation \eqref{arforgo}. Then $Q\in \Psi^{-n-1}_{\rm cl}(M;\C_+)$ is elliptic with parameter $R\in \C_+:=\{R: \mathrm{Re}(R)>0\}$. The principal symbol of $Q$ is given by 
$$\sigma_{-n-1}(Q)(x,\xi,R)=n!\omega_n(R^2+g_\rd(\xi,\xi))^{-(n+1)/2},$$
where $\omega_n$ denotes the volume of the $n$-dimensional unit ball.

In local coordinates, the full symbol of $Q$ admits a classical asymptotic expansion $ \sum_{j=0}^\infty q_j$ where the $-n-1-j$-homogeneous functions $q_j\in C^\infty((T^*M\oplus \C_+)\setminus (M\times \{0\}))$ are given by $q_0= \sigma_{-n-1}(Q)(x,\xi,R)$ and for $j>0$, $q_j$ is in coordinates given by 
\small
\begin{align*}
q_j(x,\xi,R)=&
\begin{cases} 
&\sum_{\gamma\in I_j, \mathrm{rk}(\gamma)<(n+1)/2} \mathfrak{c}_{\mathrm{rk}(\gamma),n}C^{(\gamma)}(x,-D_\xi)(R^2+g_\rd(\xi,\xi))^{-(n+1)/2+\mathrm{rk}(\gamma)}+ \qquad \mbox{for $n$ odd}\\
& -\sum_{\gamma\in I_j, \mathrm{rk}(\gamma)\geq (n+1)/2} \mathfrak{c}_{\mathrm{rk}(\gamma),n}C^{(\gamma)}(x,-D_\xi)\left[(R^2+g_\rd(\xi,\xi))^{-(n+1)/2+\mathrm{rk}(\gamma)}\log(R^2+g_G(\xi,\xi))\right],\\
{}\\
&\sum_{\gamma\in I_j} \mathfrak{c}_{\mathrm{rk}(\gamma),n}C^{(\gamma)}(x,-D_\xi)(R^2+g_\rd(\xi,\xi))^{-(n+1)/2+\mathrm{rk}(\gamma)}, \ \mbox{for $n$ even.}
\end{cases}
\end{align*}
\normalsize
The coefficients $\mathfrak{c}_{k,n}$ are given by
$$\mathfrak{c}_{k,n}:= 
\begin{cases}
(-1)^k(n-2k)!\omega_{n-2k}\omega_{2k}, \; &\mbox{for $2k<n$}\\
\frac{(-1)^{1-n/2} \omega_{2k}}{(2k-n)!\omega_{2k-n}}, \; & \mbox{for $2k-n\in 2\N$}\\
\frac{(-1)^{\frac{n+1}{2}}}{(2\pi)^{2k-n}}\omega_{2k}\omega_{2k-n-1} , \; &\mbox{for $2k- n\in 2\N+1$}
\end{cases}$$
\end{thm}

\begin{proof}[Sketch of proof]
For the full proof, see \cite[Theorem 2.9]{gimpgofflouc}. We follow the notation of \cite{gimpgofflouc}. The ideas in the proof rely on elementary techniques of calculus, and we recall the salient features. We extend $Q$ to a function of $R\in \C\setminus i\R$ by declaring $Q$ to be even in $R$. By taking a Fourier transform in the $R$-direction, the Schwarz kernel of $Q$ is transformed to the distribution 
$$K(x,y,\eta)=\chi(x,y)K_0(x,y,\eta), \quad \mbox{where} \quad K_0(x,y,\eta):=-\log(\eta^2+g_0(x,y)).$$
The statement of the theorem is local in nature, so it suffices to compute with $K_0$ in local coordinates. Standard considerations show that $K_0$ is a conormal distribution $K_0\in CI^{-n-1}(Z;\mathrm{Diag}_M\times \{0\})$ where $Z=U\times \R$ is a neighborhood of $\mathrm{Diag}_M\times \R\subseteq M\times M\times \R$. The arguments in \cite[Theorem 2.9]{gimpgofflouc} ensures that $Q$ indeed is a pseudodifferential operator with parameter of order $-n-1$ with the prescribed principal symbol. 

The computationally delicate issue is that of finding the full symbol in local coordinates. We do so by expanding $K_0$ near $x=y$ and inverse transform in $(x-y,\eta)$ to the symbol depending on $(\xi,R)$. Using the Taylor expansion \eqref{taylor}, with $v=x-y$, we write 
$$K_0(x,y,\eta)=-\log(\eta^2+H_{\rd}(v,v))-\log\left(1+\frac{\sum_{j=3}^NC^{(j)}(x,v)+r_N(x,v)}{\eta^2+H_{\rd}(x,v)}\right).$$
For small $v=x-y$, we can Taylor expand 
$$K_0(x,y,\eta)=-\log(\eta^2+H_{\rd}(x,v))+\sum_{j=1}^N\sum_{\gamma\in I_{j}} \frac{(-1)^{\mathrm{rk}(\gamma)+1}}{\mathrm{rk}(\gamma)}\frac{C^{(\gamma)}(x,v)}{(\eta^2+H_{\rd}(x,v))^{\mathrm{rk}(\gamma)}}+\tilde{r}_N(x,v,\eta).$$
Each term in the second sum $\sum_{\gamma\in I_{j}} \frac{(-1)^{\mathrm{rk}(\gamma)+1}}{\mathrm{rk}(\gamma)}\frac{C^{(\gamma)}(x,v)}{(\eta^2+H_{\rd}(x,v))^{\mathrm{rk}(\gamma)}}$ is homogeneous of degree $j$. The error term $\tilde{r}_n$ is controlled in \cite{gimpgofflouc}. To compute the full symbol we now compute the inverse Fourier transform in $(v,\eta)$. For $l>0$, denote the Fourier transform of $(\eta^2+H_{\rd}(v,v))^{-l}$ in the $(v,\eta)$-direction by $F_l(x,\xi,R)$. Using homogeneity and rotational invariance, we have that
\begin{align*}
F_l(x,\xi,R):=&\int_{T_xM\oplus \R} \frac{\e^{-i\xi.v-iR\eta}}{(\eta^2+H_{\rd}(v,v))^{l}} \mathrm{d}v\mathrm{d}\eta=\\
=&\begin{cases}
\kappa_{n,l} (R^2+g_\rd(\xi,\xi))^{-\frac{n+1}{2}+l}, \quad &2l-n-1\notin 2\N,\\
(R^2+g_\rd(\xi,\xi))^{-\frac{n+1}{2}+l}(\kappa_{n,l}\log(R^2+g_\rd(\xi,\xi))+\beta_{n,l}), \quad &2l-n-1\in 2\N,
\end{cases}
\end{align*}
for suitable constants $\kappa_{n,l},\beta_{n,l}\in \R$ that are explicitly computed in \cite{gimpgofflouc}. It follows that for $\gamma\in I_j$, the Fourier transform of the term
$$\frac{C^{(\gamma)}(x,v)}{(\eta^2+H_{\rd}(v,v))^{l}},$$
in the $(v,\eta)$-direction is given by $C^{(\gamma)}(x,-D_\xi)F_l(x,\xi,R)$. 
A long computation putting all of these terms together gives the full symbol computation.
\end{proof}

Theorem \ref{conomkmlog} has some rather direct consequences on distance functions that are smooth off the diagonal. We note the following consequence.

\begin{cor}
\label{penfpefnaon}
Let $M$ be an $n$-dimensional compact manifold with a distance function $\rd$ regular at the diagonal and smooth off-diagonally, e.g.~the subspace distance on a submanifold of Euclidean space. Set $\mu:=(n+1)/2$. Then $\mathcal{Z}$ is an elliptic pseudodifferential operator with parameter $R$ and order $-n-1$, and $\mathcal{Z}-Q$ is smoothing with parameter $R>0$. In particular, for a suitable $R_0>0$ and $R>R_0$ the operator 
$$\mathcal{Z}(R):H^{-\mu}(M)\to H^\mu(M),$$
is invertible. Moreover,
$$\|\mathcal{Z}(R)^{-1}-Q(R)^{-1}\|_{H^\mu(M)\to H^{-\mu}(M)}=O(R^{-\infty}),$$
and 
\begin{equation}
\label{mmmmmmmm}
\mathcal{M}_M(R)=R^{-1}\langle Q(R)^{-1}1,1\rangle+O(R^{-\infty}).
\end{equation}
\end{cor}

\begin{proof}[Sketch of proof]
If $\rd$ is smooth off-diagonally, then 
$$L(R)f(x):=\frac{1}{R}\int_M(1-\chi(x,y))\e^{-R\rd(x,y)}f(y)\mathrm{d}y,$$
is smoothing and exponentially decaying as $R\to +\infty$. We conclude that $Z=Q+L$ is an elliptic pseudodifferential operator with parameter $R$ and order $-n-1$ from Theorem \ref{conomkmlog}. Invertibility of $\mathcal{Z}$ for large enough $R$ follows from the ellipticity with parameter of $Q$. And indeed, since $\mathcal{Z}-Q$ is smoothing with parameter, so is $\mathcal{Z}^{-1}-Q^{-1}$ and the norm estimate follows. Therefore, 
$$\langle \mathcal{Z}(R)^{-1}1,1\rangle=\langle Q(R)^{-1}1,1\rangle+O(R^{-\infty}).$$
The equality \eqref{mmmmmmmm} follows as in Corollary \ref{corofofrdlforx}.
\end{proof}

\begin{remark}\label{MRremark}
The assumption in Corollary \ref{penfpefnaon} that the distance function is smooth off-diagonally can be weakened to the distance function having property $(MR)$.
\end{remark}

\subsection{Asymptotic expansions for compact manifolds}
\label{asymtppomomad}

For a compact manifold, Corollary \ref{penfpefnaon} provides means of computing the asymptotics of the magnitude function. Starting from \cite[Subsection 2.5]{gimpgofflouc} we compute the asymptotics of $\langle Q(R)^{-1}1,1\rangle$ as $R\to \infty$ using semiclassical analysis of the pseudodifferential operator with parameter $Q(R)^{-1}$. We mention two instrumental results in this direction. 

\begin{lem}
\label{lasdnaodjna}
Consider a properly supported pseudodifferential operator with parameter $A\in \Psi^m_{\rm cl}(M;\Gamma)$ on a manifold $M$. Setting 
$$ a_{j,0}(x,R):=a_j(x,0,R),$$
in each coordinate chart where $\sum_j a_j$ is a homogeneous expansion of the full symbol of $A$ in that chart, produces a sequence $(a_{j,0})_{j\in \N}\subseteq C^\infty(M\times \Gamma)$ of functions such that 
\begin{enumerate}
\item Each $a_{j,0}=a_{j,0}(x,R)$ is homogeneous of degree $m-j$ in $R$.
\item For any $N\in \N$, we have that
\begin{equation}
\label{expansisonandao}
[A(R)1](x)=\sum_{j=0}^N a_{j,0}(x,R)+r_{N}(x,R)=\sum_{j=0}^N a_j(x,1)R^{m-j}+r_{N}(x,R),
\end{equation}
where $r_N\in C^\infty(M\times \Gamma)$ is a function such that for any compact $K\subseteq M$ it holds that
$$\sup_{x\in K}|\partial_x^\alpha\partial_R^kr_N(x,R)|=O(\mathrm{Re}(R)^{m-N+|\alpha|+k}), \quad\mbox{as $\mathrm{Re}(R)\to +\infty$}.$$
\end{enumerate}
\end{lem}

The proof of this lemma can be found in \cite[Lemma 2.24]{gimpgofflouc}. One way to see why it is true, is to note that it is a local statement and can as such be reduced to a statement for compactly supported pseudodifferential operators with parameter in $\R^n$. After this reduction, it is a short computation with the Fourier transform that 
$$a(x,D,R)1=a(x,0,R).$$

\begin{prop}
\label{alksdnaksndasodn}
Let $M$ be an $n$-dimensional compact manifold with a distance function $\rd$ regular at the diagonal and smooth off-diagonally, e.g. the subspace distance on a submanifold of Euclidean space. Set $\mu:=(n+1)/2$. Then for a suitable $R_0>0$ and $R>R_0$ the operator 
$$\mathcal{Z}(R)^{-1}:H^{\mu}(M)\to H^{-\mu}(M),$$
is an elliptic pseudodifferential operator with parameter of order $n+1$. In each coordinate chart, we can compute the full symbol of $\mathcal{Z}(R)^{-1}$ as an asymptotic sum $\sum_k a_k$ where $a_0,a_1,\ldots$ are computed iteratively from 
$$a_0(x,\xi,R)=\frac{1}{n!\omega_n}(R^2+g_\rd(\xi,\xi))^{(n+1)/2},$$
and 
$$a_k=-a_0\sum_{|\alpha|+j+l=k, l<k} \frac{1}{\alpha!} \partial_\xi^\alpha q_kD_x^\alpha a_l.$$
\end{prop}

The proof of this proposition and further details on the precise form of the symbols can be found in \cite[Corollary 2.21]{gimpgofflouc}. The construction is standard in pseudodifferential calculus and produces a parametrix that by abstract nonsense reproduces $\mathcal{Z}(R)^{-1}$ up to operators smoothing with parameter. The crucial feature of Proposition \ref{alksdnaksndasodn} is that it produces a way of explicitly computing an inverse to arbitrary level of precision -- where the precision manifests itself through Lemma \ref{lasdnaodjna} as to which order of $R$ we can compute $\langle \mathcal{Z}(R)^{-1}1,1\rangle$.

\begin{thm}
\label{magcompsclosedeld}
Let $M$ be an $n$-dimensional compact manifold with a distance function $\rd$ whose square is regular at the diagonal. Let $(a_{j,0})_{j\in \N}\subseteq C^\infty(M;\C_+)$ denote the sequence of homogeneous functions obtained from restriction to $\xi=0$ of the full symbol of $Q_M^{-1}$, as in \cite[Section 2.5]{gimpgofflouc}. It holds that 
$$\langle 1,Q_M(R)^{-1}1\rangle =\frac{1}{n!\omega_n} \sum_{k=0}^\infty c_k(M,\rd)R^{n+1-k}+O(\mathrm{Re}(R)^{-\infty}), \quad \mbox{as $\mathrm{Re}(R)\to +\infty$},$$
where 
$$c_k(M,\rd)=n!\omega_n\int_M a_{k,0}(x,1)\rd x.$$
Here $\rd x$ is the Riemannian volume density defined from $g_{\rd^2}$. The functions $a_{k,0}(x,1)$ are explicitly constructed inductively from the Taylor expansion \eqref{taylor} using Lemma \ref{lasdnaodjna} and Proposition \ref{alksdnaksndasodn}. In particular, 
\begin{align*}
c_k(M,\rd)=
\begin{cases}
0,\; &\mbox{when $k$ is odd,}\\
\mathrm{vol}(M,g_\rd),\; &\mbox{when $k=0$},\\
\frac{n+1}{6}\int_X s_{\rd^2}\rd x,\; &\mbox{when $k=2$},
\end{cases}
\end{align*}
where $s_{\rd^2}$ in local coordinates is computed as the polynomial in the Taylor coefficients of $\rd^2$ at the diagonal given as
\begin{align*}
s_{\rd^2}(x):=&3C^4(x,g\otimes g) - 3\frac{\mathfrak{c}_{2,n}(n+5)(n^2-9)}{\mathfrak{c}_{1,n}}(C^3\otimes C^3)(x,g\otimes g \otimes g), \quad\mbox{if $n\neq 1,3$}\\
s_{\rd^2}(x):=&3\bigg(10C^{4}_G(x,g_G\otimes g_G)- \frac{\mathfrak{c}_{2,3}}{\mathfrak{c}_{1,3}}(C^{3}_G\otimes C^{3}_G)(x,g_G\otimes g_G\otimes g_G)\bigg), \quad\mbox{if $n=3$}\\
\end{align*}
\end{thm} 

Theorem \ref{magcompsclosedeld} follows by combining Lemma \ref{lasdnaodjna} and Proposition \ref{alksdnaksndasodn}. For more details on the structure of the local densities $a_{k,0}$, see \cite[Subsection 2.5 and Theorem 6.1]{gimpgofflouc}. The notation $s_{\rd^2}$ in Theorem \ref{magcompsclosedeld} is justified by \cite[Example 2.30]{gimpgofflouc}, which shows that in the case of the geodesic distance function on a Riemannian manifold $s_{\rd^2}$ is the scalar curvature. The following theorem is a direct consequence of Corollary \ref{corofofrdl}, Corollary \ref{penfpefnaon} and Theorem \ref{magcompsclosedeld}.

\begin{thm}
\label{magcompsclosedeldforz}
Let $M$ be an $n$-dimensional compact manifold with a distance function $\rd$ with property (MR) on $\Gamma$, e.g. if $\rd$ is regular at the diagonal and smooth off-diagonally. It holds that 
$$\mathcal{M}_M(R)=\frac{1}{n!\omega_n} \sum_{k=0}^\infty c_k(M,\rd)R^{n-k}+O(\mathrm{Re}(R)^{-\infty}), \quad \mbox{as $\mathrm{Re}(R)\to +\infty$ in $\Gamma$},$$
where $c_k(M,\rd)$ is as in Theorem \ref{magcompsclosedeld}. The structure of the coefficients $c_j$ is as follows:
\begin{itemize}
\item $c_0(M,\rd)=\mathrm{vol}_n(M,g_\rd)$.
\item For $j\geq 1$, the coefficient $c_{j}(M,\rd)$ is an integral over $M$ of a polynomial (universally constructed from the Taylor expansion \eqref{taylor} of $g_0:=\rd^2$ at the diagonal) in the covariant derivatives of the curvature of $M$ where the total degree is $\leq j$.
\end{itemize}
\end{thm}

We also have the following corollary, which follows from considerations in \cite{gimpgofflouc} even for manifolds which do not satisfy property (MR).

\begin{thm}
\label{compactthm}
Let $M$ be a compact $n$-dimensional Riemannian manifold equipped with its geodesic distance $\rd$. Let $g$ denote its Riemannian metric and $\rd y$ the associated volume density. Assume that $M$ admits a family $(u_R)_{R>R_0}\subseteq \mathcal{D}'(M)$ of distributional solutions to 
$$\int_M \e^{-R\rd(x,y)}u_R(y)\rd y=1, \quad R>R_0.$$
Assume that for any $\chi,\chi'\in C^\infty(M)$ with disjoint supports it holds that 
$$\int_M \chi(x)\e^{-R\rd(x,y)}u_R(y)\chi'(y)\rd y=O(R^{-N}).$$ 
Then the magnitude function $\mathcal{M}_M(R)$ is defined for $R>R_0$ and there is an asymptotic expansion  
\begin{align*}
n!\omega_n\mathcal{M}_M(R) =&\mathrm{vol}_n(M,g)R^{n}+\frac{n+1}{6} \int_Ms \rd xR^{n-2}+\\
&+ \sum_{k=4}^{N-1}c_k(M)R^{n-k}+ O(\mathrm{Re}(R)^{n-N}), \quad \mbox{as $\mathrm{Re}(R)\to +\infty$},
\end{align*}
where $s$ denotes the scalar curvature of $M$,  $c_j(M)=0$ if $j$ is odd and for $j\geq 4$, the coefficient $c_{j}$ is an integral over $M$ of a universal polynomial in the covariant derivatives of the curvature of $M$ where the total degree is $\leq j$.
\end{thm}
\begin{remark}
\label{compactthmrmk}
We note that the assumption of Theorem \ref{compactthm} is automatically satisfied if the distance function $\rd$ satsfies property (MR). In this case we can take $u_R=\mathcal{Z}(R)^{-1}1$, which is smooth by elliptic regularity, using an argument analogous to that in Proposition \ref{alksdnaksndasodn}. The authors of this paper do not know of a Riemannian manifold not satisfying the assumption of Theorem \ref{compactthm}. Indeed, the examples in  \cite[Section 3]{gimpgofflouc} of Riemannian manifolds which do not satisfy property (MR) do satisfy the assumption of Theorem \ref{compactthm} by the results of \cite{will}.
\end{remark}

\begin{proof}[Proof of Theorem \ref{compactthm}]
The assumptions on the existence of the distributional solution implies that $\mathcal{M}_M(R)=\int_M u_R(y)\rd y$ is defined for $R>R_0$. By \cite[Proposition 3.2]{gimpgofflouc}, we have that 
$$\mathcal{M}_M(R)= \frac{1}{n!\omega_n}\sum_{k=0}^N c_k(M)R^{n-k}+O(R^{n-N}),$$ 
where the coefficients are computed as in Theorem \ref{magcompsclosedeld}. 

By well known invariant theory, $c_j(M)$ is an integral over $M$ of a universal polynomial in the covariant derivatives of the curvature of $M$ where the total degree is $\leq j$. The only possible invariants in degree $0$ is the volume and in degree $2$ it is the integral of the scalar curvature. The universality implies that $c_0(M)=\alpha_0\mathrm{vol}_n(M,g)$ and $c_2(M)=\alpha_2\int_M s\mathrm{d}V$ for dimensional constants $\alpha_0$ and $\alpha_2$. The values $\alpha_0=1$ and $\alpha_2=\frac{n+1}{6}$ can be read of from works of Willerton \cite[Theorem 11]{will}. Since the Taylor expansion of the geodesic distance function near the boundary only contains even degree terms, $c_j(M)=0$ if $j$ is odd. 
\end{proof}

\section{The structure of the magnitude function for manifolds with boundary}\label{sec:boundary}

Let $X$ be an $n$-dimensional manifold with boundary equipped with a distance function. We assume that $X$ is a domain with smooth boundary in a manifold $M$. Using a variation of Theorem \ref{symbcorboundary}, we know that 
$$\mathcal{Z}_X(R):\dot{H}^{-\mu}(X)\to \bar{H}^\mu(X),$$
is an isomorphism for sufficiently large $R$, as soon as $\rd$ has property $(MR)$, see \cite[Theorem 4.7]{gimpgofflouc}. This occurs for instance when $\rd$ is regular at the diagonal and smooth off-diagonally. By a variation of Corollary \ref{corofofrdlforx}, we know how to compute $\mathcal{M}_X$ from $\mathcal{Z}_X(R)^{-1}$. The problem is to describe the inverse $\mathcal{Z}_X(R)^{-1}$ in the presence of boundary. This issue was solved in  \cite{gimpgofflouc} using ideas of Wiener-Hopf factorization dating back to Eskin \cite{eskinbook} and H\"{o}rmander \cite{hornotes}. We first give a brief overview of the computational tools entering into this construction, after which we compute the asymptotics of the magnitude function.

\subsection{Wiener-Hopf factorization and inverting $\mathcal{Z}_X$}
\label{wieninsubses}

In order to invert $\mathcal{Z}_X$ we use standard ideas of parametrix constructions in pseudodifferential calculus: to invert $\mathcal{Z}_X$ we invert it in the interior and near the boundary. By the arguments in the preceding sections, it suffices to construct $Q_X^{-1}$ under property $(MR)$. We shall see that the inverse takes the form 
$$Q_X^{-1}=\chi A\tilde{\chi}+\mathcal{W}_X+S,$$
where $A$ is a pseudodifferential parametrix with parameter (constructed in the same way as in Proposition \ref{alksdnaksndasodn}) to the localized operator $Q$, $\chi,\tilde{\chi}\in C^\infty_c(X^\circ)$ are cut off functions, $\mathcal{W}_X$ is constructed from an inversion procedure near the boundary and finally $S$ is a remainder term whose contribution to the magnitude asymptotics is negligible. 

Let us describe how to construct the operator $\mathcal{W}_X$ near the boundary. By localizing the problem, we can consider an associated model operator 
$$Q^\partial:\dot{H}^{-\mu}(\partial X\times [0,\infty))\to \bar{H}^\mu(\partial X\times [0,\infty)).$$
The operator $Q^\partial$ is constructed as coinciding with $Q_X$ in a collar neighborhood of the boundary and extended to $\partial X\times [0,\infty)$ as a fractional Laplacian so that $\sigma_{-n-1}(Q^\partial)$ is of the form in Theorem \ref{conomkmlog} and $Q^\partial$ commutes with translations outside a compact subset. The problem will be to 
\begin{enumerate}
\item Up to suitable errors, factor $Q^\partial$ as operators 
\begin{align}
\label{facomonarain}
\dot{H}^{-\mu}(\partial X\times [0,\infty))\xrightarrow{Q^\partial_+}&\dot{H}^{0}(\partial X\times [0,\infty))= \\
\nonumber
=&\bar{H}^0(\partial X\times [0,\infty))\xrightarrow{Q^\partial_-}  \bar{H}^\mu(\partial X\times [0,\infty)),
\end{align}
where $Q^\partial_\pm$ are operators of order $-\mu$. For this factorization to make sense, we need that $Q^\partial_+$ comes from an operator on $\partial X\times \R$ that preserves support in $\partial X\times [0,\infty)$ and similarly that $Q^\partial_-$ comes from an operator on $\partial X\times \R$ that preserve supports in $\partial X\times (-\infty,0]$. 
\item For the factorization to produce a useful outcome, we need that $Q^\partial_\pm$ have suitable ellipticity properties, and their parametrices $W^\partial_\pm$ also need the same domain preservation properties. If this is the case, we can invert $Q^\partial$ up to suitable error terms as the composition
\begin{align}
\label{facomonaraiwithwn}
\bar{H}^\mu(\partial X\times [0,\infty))&\xrightarrow{W^\partial_-} \bar{H}^0(\partial X\times [0,\infty))=\\
\nonumber
=&\dot{H}^{0}(\partial X\times [0,\infty))\xrightarrow{W^\partial_+}  \dot{H}^{-\mu}(\partial X\times [0,\infty)).
\end{align}
\end{enumerate}

A well known, yet crucial, observation to solve these two problems is that support preservation in a half-space is characterized by Paley-Wiener's theorem. We write coordinates on $\partial X\times [0,\infty)$ as $(x',x_n)$ and cotangent directions as $(\xi',\xi_n)$. In light of Paley-Wiener's theorem, an operator $A=a(x,D)$ preserving supports in $\partial X\times [0,\infty)$ is characterized by the symbol $a=a(x,\xi',\xi_n)$ having a holomorphic extension in $\xi_n$ to the lower half-plane. 

The first item in the list above is by Paley-Wiener's theorem solved by Wiener-Hopf factorization of the full symbol of $Q^\partial$ into factors holomorphically invertible in the upper and lower half-plane respectively. A technical issue that arises is that the factorization takes place in a space of \emph{mixed-regularity symbols}, for more details see \cite[Section 5]{gimpgofflouc}. This issue is visible when we factor the principal symbol as
$$n!\omega_n(R^2+g_\rd(\xi,\xi))^{-\mu}=q_{0,+}(x,\xi',\xi_n)q_{0,-}(x,\xi',\xi_n),$$
where 
\begin{equation}
\label{qomnomom}
\begin{cases}
q_{+,0}(x,\xi',\xi_n)&=n!\omega_n(\xi_n-h_+(x,\xi',R))^{-\mu}\\
q_{-,0}(x,\xi',\xi_n)&=h_0(x)^{-\mu}(\xi_n-h_+(x,\xi',R))^{-\mu}.
\end{cases}
\end{equation}
Here $h_0$ is a smooth function and $h_\pm$ are smooth functions (for $\xi'\neq 0$) of degree $+1$ in $\xi'$. The functions $h_0$ and $h_\pm$ come from the polynomial factorization  
$$R^2+g_\rd(\xi,\xi)=h_0(x)(\xi_n-h_+(x,\xi',R))(\xi_n-h_-(x,\xi',R)),$$
with $h_+(x,\xi',R)$ being the complex root in the upper half-plane to the second order polynomial equation $R^2+g_\rd(\xi',\xi_n)=0$ in $\xi_n$. Similarly, $h_-(x,\xi',R)$ is the complex root in the lower half-plane. The functions $q_{0,\pm}$ are not ordinary H\"{o}rmander symbols since derivatives in the $\xi'$ direction generally only improve symbol decay in the $\xi'$-direction (and not in the $\xi_n$-direction). The symbols appearing rather satisfy a mixed-regularity symbol estimate described in \cite[Section 5.1]{gimpgofflouc}. The theory of these mixed-regularity symbols seems known to experts and can be found reviewed in \cite[Section 5.1]{gimpgofflouc}. A crucial feature is that the theory of mixed-regularity symbols fits well with parameter dependence. In particular, operators of mixed-regularity $(m,-\infty)$ have operator norm $O(R^{-\infty})$ as operators between Sobolev spaces $H^s\to H^{s-m}$. Therefore, for the purpose of obtaining asymptotics this technical issue has little impact on the output. 

By an abuse of notation, we write $q\sim \sum_j q_j$ for the full symbol of $Q^\partial$ in a coordinate chart so each $q_j$ is computed near $\partial X\times \{0\}$ as in Theorem \ref{conomkmlog}. We have the following result.

\begin{thm}[Wiener-Hopf factorization]
\label{whopffaad}
There exists mixed-regularity symbols $q_+,q_-\in S^{-\mu,0}(\partial X\times \R;\C_+)$ with parameter $R\in \C_+$ such that 
\begin{enumerate}
\item We have that 
$$q\sim \sum_\alpha \frac{1}{\alpha!} \partial_\xi^\alpha q_-D_x^\alpha q_+,$$
up to terms in $S^{2\mu,-\infty}(\partial X\times \R;\C_+)$.
\item The mixed-regularity symbols $q_+,q_-\in S^{-\mu,0}(\partial X\times \R;\C_+)$ admits asymptotic expansions
$$q_\pm \sim  \sum_{j=0}^\infty q_{\pm,j},$$
up to terms in $S^{-\mu,-\infty}(\partial X\times \R;\C_+)$, where $q_{\pm,j}$ can be constructed inductively by a partial fraction decomposition from $q\sim \sum_j q_j$, and $q_{\pm,0}$ is as in Equation \eqref{qomnomom}.
\item The symbols $q_{\pm,j}$ extend holomorphically to the half-plane $\pm \mathrm{Im}(\xi_n)>0$ and for $j>0$ can be written as 
.$$q^\partial_{\pm, j}(x,\xi,R)=\sum_{k=-1}^{j-1} b_{\pm, j,k}(x,\xi',R) (\xi_n-h_\pm(x,\xi',R))^{-\mu-j+k}\in S^{-\mu-1,-j+1},$$ 
where $b_{\pm, j,k}$ is homogeneous of degree $-k$ in $(\xi',R)$.
\end{enumerate} 
In particular, the operators $Q_\pm^\partial :=Op(q_\pm)$ satisfies that $Q^\partial_+$ preserve supports in $\partial X\times [0,\infty)$, $Q^\partial_-$ preserve supports in $\partial X\times (-\infty,0]$ and the two operators factorize $Q^\partial$ as in Equation \eqref{facomonarain} up to an operator of mixed-regularity $(2\mu,-\infty)$, i.e. $Q^\partial-Q_-^\partial Q_+^\partial\in \Psi^{2\mu,-\infty}(\partial X\times \R;\C_+)$.
\end{thm}

The idea of Wiener-Hopf factorization goes gack to Eskin. The reader can find further details in the general case in H\"{o}rmander \cite{hornotes} or Grubb \cite[Theorem 2.7]{grubbibp}. Let us briefly indicate how the construction of the terms $q_{\pm,j}$ goes in the special case of interest to this work. The identity $q\sim \sum_\alpha \frac{1}{\alpha!} \partial_\xi^\alpha q_-D_x^\alpha q_+$ is equivalent to requiring 
$$\frac{q^\partial_{+,j}}{q^\partial_{+,0}}+\frac{q^\partial_{-,j}}{q^\partial_{-,0}}=\frac{q^\partial_j}{q^\partial_0}-\frac{1}{q^\partial_0}\sum_{\substack{k+l+|\alpha|=j\\k,l<j}}\frac{1}{\alpha!} \partial_\xi^\alpha q_{-,k}^\partial D^\alpha_x q_{+,l}^\partial$$
Starting from \eqref{qomnomom} and the identity $q\sim \sum_\alpha \frac{1}{\alpha!} \partial_\xi^\alpha q_-D_x^\alpha q_+$ we can proceed to inductively determine $q_{\pmb,j}$ from  $q_{\pm, k}$ for $k<j$ by performing a partial fraction decomposition
\begin{equation}
\label{ofomomfd}
\frac{q^\partial_j}{q^\partial_0}-\frac{1}{q^\partial_0}\sum_{\substack{k+l+|\alpha|=j\\k,l<j}}\frac{1}{\alpha!} \partial_\xi^\alpha q_{-,k}^\partial D^\alpha_x q_{+,l}^\partial=\mathfrak{q}_{+,j}+\mathfrak{q}_{-,j}.
\end{equation}
Indeed, a careful analysis of the left hand side in \eqref{ofomomfd} shows that it is a rational function and all denominators are products involving $\xi_n-h_+$ and $\xi_n-h_-$. As such, a partial fraction decomposition gives that
$$\mathfrak{q}_{\pm, j}(x,\xi,R)=\sum_{k=-1}^{j-1} \mathfrak{b}_{\pm, j,k}(x,\xi',R) (\xi_n-h_\pm(x,\xi',R))^{-j+k}\in S^{-1,-j+1},$$
where $\mathfrak{b}_{\pm, j,k}$ is homogeneous of degree $-k$ in $(\xi',R)$ and can be explicitly computed. We now define 
$$q^\partial_{\pm,j}:=q^\partial_{\pm,0}\mathfrak{q}_{\pm, j}.$$
The reader should note that the Wiener-Hopf factorization relies on the decomposition in \eqref{ofomomfd}; for more general symbols (such as in \cite{hornotes}) the decomposition requires the Cauchy integral that leads to less explicit factorizations. The reader can consult \cite[Section 5]{gimpgofflouc} for an explicit computation of $q_{\pm,1}$. 

What is of interest for us is that the leading symbols in the operators $Q_\pm^\partial$, namely $q_{\pm,0}$, are nowhere vanishing in their domains of holomorphicity. We can therefore proceed in the usual way with a parametrix construction.

\begin{prop}
\label{invertinthefactorso}
The operators 
$$\dot{H}^{-\mu}(\partial X\times [0,\infty))\xrightarrow{Q^\partial_+}\dot{H}^{0}(\partial X\times [0,\infty)),\quad\mbox{and}\quad\bar{H}^0(\partial X\times [0,\infty))\xrightarrow{Q^\partial_-}  \bar{H}^\mu(\partial X\times [0,\infty)),$$
are invertible for large enough $R$. Up to a operators of mixed-regularity $(-1,-\infty)$, the inverse of $Q^\partial_\pm$ can be computed as $W_\pm^\partial:=Op(w_\pm)$ where the mixed-regularity symbols $w_+,w_-\in S^{\mu,0}(\partial X\times \R;\C_+)$ admits asymptotic expansions
$$w_\pm \sim  \sum_{j=0}^\infty w_{\pm,j},$$
up to terms in $S^{\mu,-\infty}(\partial X\times \R;\C_+)$, where $w_{\pm,j}$ can be constructed inductively from $q_\pm\sim \sum_j q_{\pm,j}$ as
$$w_{\pm,0}(x,\xi,R):=(q^\partial_{\pm,0})^{-1}=
\begin{cases} 
\frac{1}{n!\omega_n} (\xi_n-h_+(x,\xi',R))^\mu, \;&\mbox{for $+$},\\
{}\\
h_0(x)^\mu(\xi_n-h_-(x,\xi',R))^\mu, \;&\mbox{for $-$},\end{cases}$$
and  
$$w_{\pm,j}:=-w_{\pm, 0}\sum_{k+l+|\alpha|=j, \, l<j}\frac{1}{\alpha!} \partial_\xi^\alpha q_{\pm, k}^\partial D_x^\alpha w_{\pm,l}.$$

The symbols $w_{\pm,j}$ extend holomorphically to the half-plane $\mp \mathrm{Im}(\xi_n)>0$ and for $j>0$ can be written as 
$$w^\partial_{\pm, j}(x,\xi,R)=\sum_{k=-1}^{j-1} \mathfrak{w}_{\pm, j,k}(x,\xi',R) (\xi_n-h_\pm(x,\xi',R))^{\mu-j+k}\in S^{\mu-1,-j+1},$$ 
where $\mathfrak{w}_{\pm, j,k}$ is homogeneous of degree $-k$ in $(\xi',R)$.
\end{prop}

We can now conclude that $W^\partial:=W_-^\partial W_+^\partial:\bar{H}^\mu(\partial X\times [0,\infty))\to \dot{H}^{-\mu}(\partial X\times [0,\infty))$ is well defined and inverts $Q^\partial$ up to operators of mixed-regularity $(0,-\infty)$ in the sense that 
$$1-W^\partial Q^\partial,1-Q^\partial W^\partial\in \Psi^{0,-\infty}(\partial X\times \R;\C_+).$$
Gluing the inverse of $Q^\partial$ together with the interior parametrix produces the following description of the inverse of $Q_X$.

\begin{thm}
\label{reoslsdlnadkn}
Let $X$ be a compact manifold with boundary and $\rd$ a distance function regular at the diagonal. The inverse of 
$$Q_X:\dot{H}^{-\mu}(X)\to \bar{H}^\mu(X),$$
can be written as 
$$Q_X^{-1}=\chi_1 A\tilde{\chi}_1+\chi_2 W_-^\partial W_+^\partial \tilde{\chi}_2+S,$$
where $A$ is a pseudodifferential parametrix to the localized operator $Q$ on $X$, $\chi_1,\tilde{\chi}_1,1-\chi_2,1-\tilde{\chi}_2\in C^\infty_c(X^\circ)$ are cut off functions, $W_\pm^\partial$ are constructed as in Proposition \ref{invertinthefactorso} from parametrices of the Wiener-Hopf factors $Q_\pm^\partial$ and  $S\in \Psi^{-2\mu,-\infty}(X;\C_+)$. 
\end{thm}

\subsection{Computing the asymptotics in presence of boundary}
\label{subseconasus}

The main consequence of Theorem \ref{reoslsdlnadkn} that we utilize for magnitude computations is that we can write
\begin{equation}
\label{aldnasodnaoj}
\langle Q_X^{-1}1,1\rangle_{L^2(X)}=\langle A1,\chi_1\rangle_{L^2(X)}+\langle W_-1, (W_+)^*\chi_2\rangle_{L^2(\partial X\times[0,\infty))}+O(R^{-\infty}).
\end{equation}
The right hand side of the expression \eqref{aldnasodnaoj} is particularly tractable as both $W_-$ and $(W_+)^*$ preserve supports in $\partial X\times (-\infty,0]$. Writing out the right hand side \eqref{aldnasodnaoj} in terms of the homogeneous expansion of the symbol, we can partially integrate term by term in the asymptotic sum and arrive at the following computation for the asymptotics of $\langle Q_X^{-1}1,1\rangle_{L^2(X)}$.

\begin{cor}
\label{asumomoamda}
Let $X$ be a compact manifold with boundary and $\rd$ a distance function regular at the diagonal. We have that 
\begin{align*}
\langle Q_X^{-1}1,1\rangle_{L^2(X)}=&\sum_j\left(\int_X a_{0,j}(x,1)\rd x+\int_{\partial X} B_{j,\rd^2}(x)\rd x\right)R^{n+1-j},
\end{align*}
where $a_{0,j}$ are the local densities computed from a parametrix of $Q$ as in Lemma \ref{lasdnaodjna}, and $(B_{\rd^2,j})_{j>0}\subseteq C^\infty(\partial X)$ is the sequence of functions defined by
$$B_{\rd^2,j}(x'):=\sum_{\substack{j=|\beta|+\gamma_n+k+l\\\gamma_n>0}}\frac{i^{|\beta|+|\gamma_n|}(-1)^{|\beta|+1}}{\beta'!(\beta_n+\gamma_n)!}\partial_{x}^{\beta}w_{-,k}(x',0,0,1)\partial_{x_n}^{\gamma_n-1} \partial_\xi^{\beta+(0,\gamma_n)} w_{+,l}(x',0,0,1),$$
for $j>0$ and $B_0:=0$.
\end{cor}

From the local nature of the computations of asymptotics in Corollary \ref{asumomoamda}, we can conclude the following statement for the asymptotics of magnitude. The details in the computations can be found in \cite[Subsection 6.3]{gimpgofflouc}.

\begin{thm}
\label{technicalthmsec6}
Let $X$ be a compact $n$-dimensional manifold with boundary equipped with a distance function $\rd$ with property $(MR)$, e.g. if $\rd$ is regular at the diagonal and smooth off-diagonally. The magnitude function $\mathcal{M}_X$ admits an asymptotic expansion 
$$\mathcal{M}_X(R)\ =\ \frac{1}{n!\omega_n}\sum_{j=0}^\infty c_j(X,\rd) R^{n-j}+O(R^{-\infty}),$$
as $R\to \infty$ along the positive axis. Here $\omega_n$ denotes the volume of the unit ball in $\R^n$. The structure of the coefficients $c_j$ is as follows:
\begin{itemize}
\item $c_0(X,\rd)=\mathrm{vol}_n(X,g_\rd)$ computed in the Riemannian metric $g_\rd$ induced from the transversal Hessian of the distance function at the diagonal;
\item $c_1(X)=\mu\mathrm{vol}_{n-1}(\partial X,g_\rd)$ computed in the Riemannian metric induced from $g_\rd$;
\item $c_2(X)=\frac{n+1}{6}\int_X s_\rd \mathrm{d}V+\frac{\mu^2(n-1)}{2}\int_{\partial X} H_\rd\mathrm{d}S$, where $s_\rd$ is a scalar curvature like function defined from $\rd$ and $H_\rd$ is a mean curvature like term of the boundary.
\item For $j\geq 3$, we have that 
$$c_{j}(X,\rd)=n!\omega_n\int_X a_{j,0}(x,1)\rd x+n!\omega_n\int_{\partial X} B_{\rd^2,j}(x)\rd x,$$ 
where $a_{j,0}$ is a  polynomial (universally constructed from the Taylor expansion \eqref{taylor} of $\rd^2$ at the diagonal) in the covariant derivatives of the curvature of $X$ where the total degree is $\leq j$.
\end{itemize}
\end{thm}

In the special case of smooth, compact domains $X\subseteq \R^n$, one readily computes that $a_{j,0}=0$ for $j>0$ (compare to \cite[Example 2.28]{gimpgofflouc}). We can in particular deduce the following simplified asymptotics for smooth, compact Euclidean domains. For the computation of $c_3(X)$ for odd dimensions, see \cite{gimpgoff4th}, and for $c_3(X)$ in dimension $n=2$ see Remark \ref{computersays}.

\begin{thm}
\label{Rntheorem}
Let $X\subseteq \R^n$ be a compact domain with smooth boundary equipped with Euclidean distance. Let $\mathcal{M}_X$ denote its magnitude function. Then there are constants $c_j(X)$, $j=0,1,2,\ldots$, such that 
$$\mathcal{M}_X(R)\ \sim \ \frac{1}{n!\omega_n}\sum_{j=0}^\infty c_j(X) R^{n-j},$$
as $R\to \infty$ along the positive axis. The structure of the coefficients $c_j$ is as follows:
\begin{itemize}
\item $c_0(X)=\mathrm{vol}_n(X)$ 
\item $c_1(X)=\mu\mathrm{vol}_{n-1}(\partial X)$
\item $c_2(X)=\frac{\mu^2(n-1)}{2}\int_{\partial X} H\mathrm{d}S$, where $H$ is  the mean curvature of the boundary.
\item If $n=2$ or $n$ is odd, $c_3(X)=\alpha_3(n)\int_{\partial X} H^2\mathrm{d}S$ for a dimensional constant $\alpha_3(n)>0$.
\item For $j\geq 3$, we have that 
$$c_{j}(X,\rd)=n!\omega_n\int_{\partial X} B_{\rd^2,j}(x)\rd x,$$ 
where a $B_{\rd^2,j}(x)$ is explicitly computable as in Corollary \ref{asumomoamda} and is given by a universal polynomial in covariant derivatives of  the second fundamental form of $\partial X$ of total degree $\leq j$.
\end{itemize}
\end{thm}

\section{Some remarks about finer structures}\label{sec:finestructure}

\subsection{Localization of solutions to $A1$ and boundary terms}
\label{boundarlocalsos}

Let us provide an immediate consequence of Lemma \ref{lasdnaodjna} that highlights how the solution to the equation $R\mathcal{Z}(R)U=1$ inside a manifold can be described explicitly by means of the symbol of the parametrix $A$. An interesting situation to keep in mind is when $M=X^\circ$ is the interior of a manifold with boundary. 

\begin{cor}
\label{conseqoninterior}
Let $M$ be an $n$-dimensional manifold equipped with a distance function $\rd$ whose square is smooth and $\Gamma\subseteq \C_+$ a sector with positive angle to $i\R$. Write $(a_{j,0})_{j\in \N}\subseteq C^\infty(M\times \C_+)$ for the sequence defined as in Lemma \ref{lasdnaodjna} from $Q^{-1}\in \Psi^{n+1}_{\rm cl}(M;\C_+)$.

Assume that $(u_R)_{R\in \Gamma}\subseteq \mathcal{D}'(M)$ is a family of weak solutions to 
$$\int_M \e^{-R\rd(x,y)}u_R(y)\rd y=1, \quad x\in M.$$
Then $(u_R)_{R\in \Gamma}\subseteq C^\infty(M)$ and for any $N\in \N$ and any compact $K\subseteq M$, there is a $C>0$ such that 
$$\left|u_R(x)-\sum_{j=0}^{n+N} a_{j,0}(x,1)R^{n-j}\right|\leq C(1+|R|)^{-N}, \quad\mbox{for all $x\in K$}.$$
\end{cor}

\begin{remark}
In the special case that $M=X^\circ$ is the interior of a compact domain in $\R^n$ and $\rd$ is the Euclidean distance then for any $N\in \N$ and any compact $K\subseteq M$, there is a $C>0$ such that 
$$\left|u_R(x)-\frac{R^n}{n!\omega_n}\right|\leq C(1+|R|)^{-N}, \quad\mbox{for all $x\in K$}.$$

We note that for $n$ being odd, the results of \cite{meckes} show that $u_R=\frac{1}{n!\omega_n}(R^2-\Delta)^{n+1}h_R$ for a uniquely determined $h_R\in H^{(n+1)/2}(\R^n)$ solving $(R^2-\Delta)^{n+1}h_R=0$ in $\R^n\setminus X$ and satisfying $h_R=1$ in $X^\circ$. Therefore 
$$u_R(x)=\frac{R^n}{n!\omega_n}\quad\mbox{for all $x\in X^\circ$}.$$
Since $u_R\in H^{-(n+1)/2}(\R^n)$ is supported in $X$, $u_R-\frac{R^n}{n!\omega_n}$ is supported on $\partial X$. If $n$ is even, it is not clear if $u_R-\frac{R^n}{n!\omega_n}$ is supported on $\partial X$. However, by our result $u_R-\frac{R^n}{n!\omega_n}$ tends to $0$ away from $\partial X$ faster than any polynomial as $R \to \infty$. This generalizes \cite[Theorem 5]{bunch2} to even dimensions, as relevant for a boundary detection method in data science applications.  
\end{remark}

\subsection{Taylor expansion at $R=0$}
\label{tarolmlnad}

In light of Corollary \ref{corofofrdlforx}, it is of interest to study series expansions of solutions to $R\mathcal{Z}(R)u_R=1$ near $R=0$. Indeed, meromorphicity ensures that a Taylor expansion of the magnitude function at $R=0$ determines the magnitude function completely. Recent results of Meckes \cite{meckes20} compute the derivative of the magnitude function of a convex body in $\R^n$ at $R=0$ in terms of an intrinsic volume. The following result provides a partial description of the Taylor expansion at $R=0$ under a constraint on the distance function similar to property (MR).

\begin{thm}
\label{taylorarororofoz}
Let $X$ be a compact manifold with $C^0$-boundary with a distance function $\rd$ such that $\rd^2$ is regular at the diagonal and that $L(R)$ defines a holomorphic family of operators
$$L(R):\dot{H}^{-\mu}(X)\to \overline{H}^\mu(X),$$
for $R$ in a punctured neighborhood of $R=0$, for instance if $\rd^2$ is smooth. Assume furthermore that the operator 
$$\mathcal{Z}_1:\dot{H}^{-\mu}(X)\to \overline{H}^\mu(X), \quad \mathcal{Z}_1u(x):=\int_X \rd(x,y)u(y)\rd y$$
is invertible. Then there is a unique holomorphic family $(u_R)_{|R|\leq \delta^{-1}}\subseteq \dot{H}^{-\mu}(X)$, for some $\delta>0$, which solves the equation 
$$R\mathcal{Z}(R)u_R=1.$$
The holomorphic family $(u_R)_R$ admits an expansion at $R=0$ as an absolutely norm-convergent power series $u_R=\sum_{k=0}^\infty u_kR^k\in \dot{H}^{-\mu}(X)$. The coefficients $(u_k)_k\subseteq \dot{H}^{-\mu}(X)$ are constructed from
$$u_0:=\lambda_1g, \quad\mbox{where}\quad g:=\mathcal{Z}_1^{-1}(1)\quad\mbox{and}\quad\lambda_1:=\frac{1}{\int_X g\rd x},$$
defining $u_{k+1}$ and an auxiliary parameter $\lambda_{k+2}$ from $u_0,\ldots,u_k$ and $\lambda_{k+1}$ as
$$u_{k+1}=\sum_{l=0}^{k}(-1)^{k-l} \mathcal{Z}_1^{-1}\mathcal{Z}_{k+2-l}u_l+\lambda_{k+2}g, \quad\mbox{and}\quad \lambda_{k+2}=\frac{\lambda_{k+1}-\sum_{l=0}^{k}(-1)^{k-l} \int_X \mathcal{Z}_1^{-1}\mathcal{Z}_{k+2-l}u_l\rd x}{\int_M \mathcal{Z}_1^{-1}(1)\rd x}$$
where $\mathcal{Z}_ku(x):=\frac{1}{k!}\int_X \rd(x,y)^ku(y)\rd y$, $k\geq 0$.
\end{thm}

\begin{proof}
We can decompose 
$$\mathcal{Z}_k=Q_k+L_k,$$
where $Q_k(R)u(x)=\frac{1}{k!} \int_X \chi(x,y)\rd(x,y)^ku(y)\rd y$ is an elliptic pseudodifferential operator of order $-n-k$ for odd $k$ and smoothing for even $k$ by a similar argument as in Theorem \ref{conomkmlog}. We introduce the notation $L_0=0$ and 
$$\mathcal{Z}_0u(x)=Q_0u(x):=\int_X u(y)\rd y,$$
for the projection onto the constant function, which is a smoothing operator.
Moreover, we can write 
$$R\mathcal{Z}(R)=\sum_{k=0}^\infty (-1)^kR^k\mathcal{Z}_{k},$$
and 
$$RQ(R)=\sum_{k=0}^\infty (-1)^kR^kQ_{k},\quad \mbox{and}\quad RL(R)=\sum_{k=0}^\infty (-1)^kR^kL_{k}.$$
A short argument bounding the norm on $Q_k$ with Calderon-Vaillancourt's theorem shows that the expansion for $RQ(R)$ is an absolutely convergent Taylor expansion on a neighborhood of $R=0$. Since we have assumed that $L(R)$ is a holomorphic function, the expansion for $RL(R)$ and $R\mathcal{Z}(R)$ are absolutely convergent Taylor expansions in a neighborhood of $R=0$.

We search for an absolutely norm-convergent power series solving $R\mathcal{Z}(R)u_R=1$. We make the formal ansatz $u_R=\sum_{k=0}^\infty u_kR^k$ in which $R\mathcal{Z}(R)u_R=1$ is equivalent to $L_0u_0=1$ and
\begin{equation}
\label{equaosoform}
\sum_{l=0}^k(-1)^{k-l} \mathcal{Z}_{k-l}u_l=0, \quad k>0.
\end{equation}
Note that the range of $\mathcal{Z}_0$ consists only of constants. We write $\lambda_k:=\mathcal{Z}_0u_k$, note that $\lambda_0=1$. We set $g:=\mathcal{Z}_1^{-1}(1)$.  We can rewrite Equation \eqref{equaosoform}, using that $\mathcal{Z}_1$ is invertible, as
$$u_{k-1}=\sum_{l=0}^{k-2}(-1)^{k-l} \mathcal{Z}_1^{-1}\mathcal{Z}_{k-l}u_l+\lambda_kg.$$
Therefore, $u_k$ can be determined inductively as follows. We find $u_0$ as $u_0=\lambda_1 g$, and then we determine $\lambda_1$ from $\lambda_1\mathcal{Z}_0(g)=1$. The induction step is obtained as follows: say we have determined $u_0,\ldots, u_k$ and $\lambda_0,\ldots, \lambda_{k+1}$ we then define $u_{k+1}$ as 
$$u_{k+1}=\sum_{l=0}^{k}(-1)^{k-l} \mathcal{Z}_1^{-1}\mathcal{Z}_{k+2-l}u_l+\lambda_{k+2}g,$$
where $\lambda_{k+2}$ is defined from $\lambda_{k+2}\mathcal{Z}_0(g)=\lambda_{k+1}-\sum_{l=0}^{k}(-1)^{k-l} \mathcal{Z}_0\mathcal{Z}_1^{-1}\mathcal{Z}_{k+2-l}u_l$.The fact that the Taylor expansion of $R\mathcal{Z}(R)$ converges implies that the series $u_R=\sum_{k=0}^\infty u_kR^k$ is absolutely norm convergent in a neighborhood of $R=0$.
\end{proof}

\begin{cor}
Under the assumptions of Theorem \ref{taylorarororofoz}, $\mathcal{M}_X(R)$ is holomorphic at $R=0$ where it admits the Taylor series
$$\mathcal{M}_X(R)=1+\sum_{k=1}^\infty \lambda_k R^k,$$
for the sequence $(\lambda_k)_{k>0}$ constructed in Theorem \ref{taylorarororofoz}.
\end{cor}

\begin{remark}
In Theorem \ref{taylorarororofoz}, the result relies on the right hand side of the equation $R\mathcal{Z}(R)u=1$ being the constant function. Indeed, since $R\mathcal{Z}(R)$ evaluates to the projection onto the constant functions at $R=0$, there is no holomorphic family of solutions to $R\mathcal{Z}(R)u=f$ unless $f$ is constant.
\end{remark}

\begin{problem}
For a compact domain $X\subseteq \R^n$, does it hold true that the operator 
$$\mathcal{Z}_1:\dot{H}^{-\mu}(X)\to \overline{H}^\mu(X), \quad \mathcal{Z}_1u(x):=\int_X |x-y|u(y)\rd y,$$
is invertible? Standard techniques with G\aa rding inequalities show that $\mathcal{Z}_1$ is a Fredholm operator with vanishing index. Indeed, we can write 
$$\langle \mathcal{Z}_1 u,u\rangle_{L^2(X)}=\int_{\R^n} g(\xi)|\hat{u}(\xi)|^2\rd \xi, \quad u\in \dot{H}^{-\mu}(X),$$
where $g$ is the Fourier transform of $x\mapsto |x|$. By \cite[Proposition A.1]{gimpgofflouc}, and homogeneity arguments, we have that 
$$g(\xi)=-\pi^{\frac{n-1}{2}}2^n \Gamma\left(\frac{n+1}{2}\right) \mathrm{F.P.}|\xi|^{-n-1}.$$
Therefore this problem could be susceptible to explicit quadratic form estimates. Were the problem to have a positive solution for the unit ball $X=B_n$ for odd $n$, \cite[Theorem 4]{meckes20} implies that 
\begin{equation}
\label{z1mamfaoda}
\int_{X} \mathcal{Z}_1^{-1}1\rd x=\frac{2}{V_1(X)},
\end{equation}
where $V_1(X)$ denotes the first intrinsic volume of $X=B_n$. The claim in \cite[Conjecture 5]{meckes20} is that the identity \eqref{z1mamfaoda} holds for all compact convex domains $X$.
\end{problem}

\begin{appendix}

\section{Algorithm for computing $c_j(X,\rd)$}
\label{appendixA}

We recall the expansion of $\mathcal{M}_X(R)$ from preceding chapters, given by
$$\mathcal{M}_X(R)=\sum_{k=0}^\infty c_{k}(X,\rd)R^{n+1-k}+O(R^{-\infty}),$$
provided $(X,\rd)$ satisfies the relevant assumptions. The approach in this article leads to an algorithmic procedure for computing the coefficients $c_k$, which we summarize here and detail in special geometric situations. Each stage of this procedure has been already presented and proved in \cite{gimpgofflouc}, and a pseudocode may be found in Appendix \ref{appendixB}. 

For the convenience of the reader, we repeat some preliminary computations and definitions:
$$I_j:=\{\gamma\in \cup_{k=1}^\infty \N^k_{\geq 3}: |\gamma|=j+2k\}, \, \mathrm{rk}(\gamma) = k \text{ for }\gamma \in \N^k$$
$$\mathfrak{c}_{k,n}:= 
\begin{cases}
(-1)^k(n-2k)!\omega_{n-2k}\omega_{2k}, \; &\mbox{for $2k- n-1<0$}\\
\frac{(-1)^{1-n/2} \omega_{2k}}{(2k-n)!\omega_{2k-n}}, \; & \mbox{for $2k-n-1\in 2\N+1$}\\
\frac{(-1)^{\frac{n+1}{2}}}{(2\pi)^{2k-n}}\omega_{2k}\omega_{2k-n-1} , \; &\mbox{for $2k- n-1\in 2\N$}
\end{cases},$$
$$q_{k,p}(x,R).v:=\sum_{|\alpha|=p}\partial_\xi^\alpha q_k(x,\xi,R)v^\alpha|_{\xi=0}=\partial_t^pq_k(x,tv,R)|_{t=0};$$
By $C^{(\gamma)}(x,-D_\xi)$, $\gamma\in I_j$, we denote the $|\gamma|$-th order differential operator arising from the Taylor expansion of $\rd^2:M\times M\to [0,\infty]$ (see \eqref{taylor}) and write $g_\rd$ for the dual Riemannian metric to $H_{\rd^2}$.
Using these definitions, we calculate $c_k$ in the following steps. 
\begin{enumerate}
\item \textbf{Calculate terms $q_j$ in the expansion of $q$ for $j \in [0,k]$}\label{step:q_j}\\
$q$ denotes the symbol of the operator $Q_{G,\chi}$ and it has an asymptotic expansion $q \sim \sum_j q_j$, see Theorem \ref{conomkmlog}. This is the operator we use to establish results corresponding to the interior of $X$. 

The first term is given by
$$q_0(x,\xi,R) =(R^2+g_\rd(\xi,\xi))^{-(n+1)/2} ,$$
and for $j>0$ and $n$ odd we have that 
\smaller
\begin{align*}
q_j(x,\xi,R)=&\sum_{\gamma\in I_j, \mathrm{rk}(\gamma)<(n+1)/2} \mathfrak{c}_{\mathrm{rk}(\gamma),n}C^{(\gamma)}(x,-D_\xi)(R^2+g_G(\xi,\xi))^{-(n+1)/2+\mathrm{rk}(\gamma)}+\\
&{ -}\sum_{\gamma\in I_j, \mathrm{rk}(\gamma)\geq (n+1)/2} \mathfrak{c}_{\mathrm{rk}(\gamma),n}C^{(\gamma)}(x,-D_\xi)\\
&\qquad \qquad \qquad \left[(R^2+g_\rd(\xi,\xi))^{-(n+1)/2+\mathrm{rk}(\gamma)}\log(R^2+g_G(\xi,\xi))\right],
\end{align*}
\normalsize
and for $n$ even
\begin{align*}
q_j(x,\xi,R)=&\sum_{\gamma\in I_j} \mathfrak{c}_{\mathrm{rk}(\gamma),n}C^{(\gamma)}(x,-D_\xi)(R^2+g_\rd(\xi,\xi))^{-(n+1)/2+\mathrm{rk}(\gamma)}.
\end{align*}

When $M\subseteq \R^N$ is an $n$-dimensional submanifold of $\R^N$, we determine $C^{(\gamma)}$ in the following way: 
take coordinates around some $x_0 \in M$ such that $M$ near $x_0$ is parametrized by
$$\begin{cases}
x_l=x_l,\; &l=1,\ldots,n\\
x_{l}=\varphi_l(x_1,\ldots, x_n), \; &l=n+1,\ldots,N\end{cases},$$
for some functions $\varphi_{N+1},\ldots , \varphi_n$. Writing $x = (x_1,\ldots , x_N)$, we have that
$$H_{\rd^2} (v) = |v|^2 + \sum_{l=n+1}^N \left(\nabla\varphi_l(x)\cdot v\right)^2$$
and 
$$C^j(x,v)=\sum_{l=n+1}^N\sum_{\substack{|\alpha|+|\beta|=j,\\ |\alpha|,|\beta|>0}} \frac{\partial_{x}^{\alpha}\varphi_l(x)\partial_{x}^{\beta}\varphi_l(x)}{\alpha!\beta!} v^{\alpha+\beta}, \quad j>2.$$
This is explained in \cite[Example 2.16]{gimpgofflouc}.

In the special case that $X\subset M=\R^n$ is a domain and $\rd(x,y) = |x-y|$ the Euclidean distance, $q_j = 0$ for $j>0$ and consequently $q = q_0 = (R^2 + |\xi|^2)^{-(n+1)/2}$.

\item \textbf{Calculate terms $a_j$ in the expansion of $a$ at $\xi=0$ for $j \in [0,k]$}

By $ A$ we denote the parametrix to $Q$ in the interior of $X$, and its symbol is denoted by $a$. It has an asymptotic expansion $a \sim \sum_{j} a_j$, and at $\xi=0$, we write $a_{j,0}(x,R):=a_j(x,0,R)$. As described in \cite[Lemma 2.24]{gimpgofflouc}
the first term is given by
$$a_{0,0}(x,R) =\frac{ R^{(n+1)/2}}{n!\omega_n},$$
and for $j>0$ the terms are determined inductively by
\begin{align*}
a_{j,0}(x,R) = 
\begin{cases}
0& \mbox{$j $ odd},\\
-\frac{1}{n!\omega_n}R^{n+1}\sum_{\substack{k+2l+p=j\\2l<j, \, 2|k+p}}i^pq_{k,p}(x,R).\nabla_x^pa_{j-2k,0}(x,R) & \mbox{$j$ even}.
\end{cases}
\end{align*}

In the case that $X$ is a domain in $\R^n$, $a_j = 0$ for $j>0$ since $q=q_0$, thus implying in a similar manner that $a = a_0$. 

In the case that $M$ is an $n$-dimensional submanifold of $\R^N$, we use the same parametrisation by $\varphi_{n+1},\ldots,\varphi_{N}$ as in the previous step, see \cite[Example 2.29]{gimpgofflouc}.

\item \textbf{Calculate the terms $q^\partial_j$ in the expansion of $q^\partial$ near the boundary for $j \in [0,k-1]$}\\ 
By $Q^\partial$ we denote the localization of $Q$ near the boundary of $X$, and we denote its symbol by $q^\partial$, see Subsection \ref{wieninsubses} above or \cite[Proposition 5.11]{gimpgofflouc}. 
It has an asymptotic expansion $q^\partial \sim \sum_j q_j^\partial $ where $q^\partial_j = q_j$ using the coordinates induced by localizing near the boundary.

When $X=M$ is a compact submanifold of $\R^N$ it is immediate that there are no contributions from this operator as the boundary is empty. 

For smooth domains, we construct $q_j^\partial$ via a choice of coordinates as in the pseudocode of Example \ref{domadaindinappb} in Appendix \ref{appendixB}. In the presence of a boundary when $X\subset M=\R^n$, write $x=(x',x_n) \in \R^{n-1}\times \R$ (resp. $\xi=(\xi',\xi_n)$) and determine $Q^\partial$ in the following way: fix a point $x_0$ on the boundary where the normal vector is orthogonal to the plane $x_n=0$. Without loss of generality, this point can be $x_0=0$. Pick a smooth $\varphi$ on the boundary such that $\varphi(x')<x_n$ in a neighborhood of $x_0$ that belongs to $X$. Make a change of coordinates to $(x',x_n)\mapsto (x',x_n - \varphi(x'))$. This procedure is described in more detail in \cite[Example 2.15]{gimpgofflouc}.

In these new coordinates, 
\begin{align*}
g_\rd = H_{\rd^2}^{-1}= \begin{pmatrix}
1_{n-1}&\nabla\varphi(x')\\
\nabla\varphi(x')^T& 1+|\nabla\varphi(x')|^2\end{pmatrix}
\end{align*}
and
\small
\begin{align*}
C^{(\gamma)}(x,-D_\xi)=(-1)^{|\gamma|}\prod_{l=1}^{\mathrm{rk}(\gamma)} \Bigg[&\sum_{|\alpha'|=|\gamma_l|-1} \frac{-2\partial_{x'}^{\alpha'}\varphi(x')}{\alpha'!} D_{\xi'}^{\alpha'}D_{\xi_n}\\
& \quad \quad+\sum_{\substack{|\alpha'|+|\beta'|=|\gamma_l|,\\ |\alpha'|,|\beta'|>0}} \frac{\partial_{x'}^{\alpha'}\varphi(x')\partial_{x'}^{\beta'}\varphi(x')}{\alpha'!\beta'!} D_{\xi'}^{\alpha'+\beta'}\Bigg].
\end{align*}
\normalsize

\item \textbf{Factorize $Q^{\partial}$ near the boundary into terms supported in the upper $(+)$ and lower $(-)$ half planes}

The next thing we have to do is factorize $Q^{\partial}$ as $Q_-Q_+$, where $Q_{\pm}$ is supported in the upper resp. lower half plane. We denote the symbols of $Q_\pm$ by $q_\pm$ and they have an asymptotic expansion $q_\pm  \sim \sum_j q_{\pm,j}^\partial$, see Theorem \ref{whopffaad} above or \cite[Subsection 5.2]{gimpgofflouc}. 
Start by defining
\begin{align*}
q^\partial_{+,0}:=&n!\omega_n h_0^{-\mu}(\xi_n-h_+)^{-\mu},\\
\nonumber
q^\partial_{-,0}:=&(\xi_n-h_-)^{-\mu},
\end{align*}
and construct $q^\partial_{\pm,j}$ inductively by
\begin{align}
\frac{q^\partial_j}{q^\partial_0}-\frac{1}{q^\partial_0}\sum_{\substack{k+l+|\alpha|=j\\k,l<j}}\frac{1}{\alpha!} \partial_\xi^\alpha q_{+,k}^\partial D^\alpha_x q_{-,l}^\partial=\mathfrak{q}_{+,j}+\mathfrak{q}_{-,j},\label{eq:qpm}
\end{align}
$$q^\partial_{\pm,j}:=q^\partial_{\pm,0}\mathfrak{q}_{\pm, j}.$$
Here, the terms $h_0(x)$ is the leading $\xi_n$-term in the metric and $h_\pm(x,\xi',R)$ are the zeros in the $\pm$-plane, that all in all are determined from the equation
$$R^2+ g_\rd(\xi,\xi  ) = h_0(\xi_n - h_+) (\xi_n - h_-),$$

Split the LHS of \eqref{eq:qpm} using a partial fraction decomposition into terms that have denominators that are powers of $(\xi_n - h_+)$ resp. $(\xi_n - h_-)$. The former belong to $\mathfrak{q}_{+,j}$ and the latter to $\mathfrak{q}_{-,j}$. We note here that \cite[Appendix B]{gimpgofflouc} contains general results for these decompositions.

\item \textbf{Calculate the terms $w_{\pm,j}$ in the expansion of $w_{\pm}$ near the boundary for $j \in [0,k-1]$}\\ 
By $W_{\pm}$ we denote the inverse of $Q^\partial_\pm$, see Proposition \ref{invertinthefactorso} above or \cite[Subsection 5.2]{gimpgofflouc}. We denote its symbol by  $w_{\pm} $ and it has an asymptotic expansion (see Proposition \ref{invertinthefactorso}) $w_{\pm} \sim \sum_{j=0}^\infty w_{\pm,j}$ where $w_{\pm,j}$ is constructed inductively by
$$w_{\pm,j}:=-w_{\pm, 0}\sum_{k+l+|\alpha|=j, \, l<j}\frac{1}{\alpha!} \partial_\xi^\alpha q_{\pm, k}^\partial D_x^\alpha w_{\pm,l}$$
with
$$w_{\pm,0}(x,\xi,R):=(q^\partial_{\pm,0})^{-1}=
\begin{cases} 
\frac{1}{n!\omega_n} h_0(x)^\mu(\xi_n-h_+(x,\xi',R))^\mu \;&\mbox{for $+$},\\
{}\\
(\xi_n-h_-(x,\xi',R))^\mu \;&\mbox{for $-$}.\end{cases}$$

\item \textbf{Compute the coefficient of $R^{n+1-k}$}\\
We label this coefficient by $c_k$. It is calculated  by
\begin{align*}
c_{k}=&\int_X a_{k,0}(x,1)\rd x+\int_{\partial X}B_{d^2,k}(x)\rd x'
\end{align*}
where
\begin{align*}
B_{\rd^2,k}&(x'):=\\
&=\sum_{\substack{k=|\beta|+\gamma_n+j+l\\\gamma_n>0}}\frac{i^{|\beta|+|\gamma_n|}(-1)^{|\beta|+1}}{\beta'!(\beta_n+\gamma_n)!}\partial_{x}^{\beta}w_{-,j}(x',0,0,1)\partial_{x_n}^{\gamma_n-1} \partial_\xi^{\beta+(0,\gamma_n)} w_{+,l}(x',0,0,1).
\end{align*}
The details of this calculation are in \cite[Subsection 6.2]{gimpgofflouc}. 

When $X$ has no boundary there are no boundary contributions and consequently $c_k = \int_X a_k \; \rd x$.  

On the other hand, when $X$ is a domain in $\R^n$, if $k>0$ then $c_k = \int_{\partial_X} B_{\rd^2,k}\rd x'$ as the only contribution from the interior comes from $a_0$.

\end{enumerate}

\section{Pseudocode in special examples}
\label{appendixB}

\begin{example}[$X\subset M=\R^n$ with boundary]
\label{domadaindinappb}
\label{ex:domrn}
\leavevmode \\
\textbf{Input}:
\begin{itemize}
\item $n$: dimension of $M$
\item $(x,\xi,R) \in \R^n\times \R^n\times{\R}$
\item $k$: $n+1-k$ is the degree of term to be determined
\item $\varphi(x')$: parametrise near $\partial X$
\end{itemize}
\leavevmode \\
\textbf{Definitions}:
\begin{itemize}
\item $D_t = -i\partial_t$
\item $I_j:=\{\gamma\in \cup_{k=1}^\infty \N^k_{\geq 3}: |\gamma|=j+2k\}$
\item $\mathrm{rk}(\gamma) = $length$(\gamma)$
\item $$\mathfrak{c}_{m,n}:= 
\begin{cases}
(-1)^m(n-2m)!\omega_{n-2m}\omega_{2m}, \; &\mbox{for $2m- n-1<0$}\\
\frac{(-1)^{1-n/2} \omega_{2m}}{(2m-n)!\omega_{2m-n}}, \; & \mbox{for $2m-n-1\in 2\N+1$}\\
\frac{(-1)^{\frac{n+1}{2}}}{(2\pi)^{2m-n}}\omega_{2m}\omega_{2m-n-1} , \; &\mbox{for $2m- n-1\in 2\N$}
\end{cases},$$
\item $a_0 = \frac{1}{n!\omega_n}$, $a_j = 0 \, \forall j>0$
\item 
\smaller $$C^j(x,-D\xi) = (-1)^{|\gamma|}\prod_{l=1}^{\mathrm{rk}(\gamma)} \left[\sum_{|\alpha'|=|\gamma_l|-1} \frac{-2\partial_{x'}^{\alpha'}\varphi(x')}{\alpha'!} D_{\xi'}^{\alpha'}D_{\xi_n}+\sum_{\substack{|\alpha'|+|\beta'|=|\gamma_l|,\\ |\alpha'|,|\beta'|>0}} \frac{\partial_{x'}^{\alpha'}\varphi(x')\partial_{x'}^{\beta'}\varphi(x')}{\alpha'!\beta'!} D_{\xi'}^{\alpha'+\beta'}\right]$$\normalsize
\item $g_{bdy}$: Riemannian metric near the boundary of $X$
\begin{align*}
g_{bdy} = \begin{pmatrix}
1_{n-1}&\nabla\varphi(x')\\
\nabla\varphi(x')^T& 1+|\nabla\varphi(x')|^2\end{pmatrix}
\end{align*}
\item $h_0:=1+|\nabla\varphi(x')|^2$
\item $h_+:=- \frac{\xi' \left(\nabla \varphi (x')\right)}{h_0}+ i\frac{ \sqrt{(R^2 + \xi'^2)h_0  - \left(\xi' \left(\nabla \varphi (x')\right)\right)^2}}{ h_0}$
\item $h_-:=-\frac{\xi' \left(\nabla \varphi (x')\right)}{h_0}- i\frac{ \sqrt{(R^2 + \xi'^2)h_0  - \left(\xi' \left(\nabla \varphi (x')\right)\right)^2}}{ h_0}$
\item $q^\partial_0:=(R^2+ g_{bdy})^{-(n+1)/2}$
\item $q^\partial_{+,0}:=n!\omega_n (\xi_n-h_+)^{-(n+1)/2}$
\item $q^\partial_{+,0}:=h_0^{-(n+1)/2}(\xi_n-h_-)^{-(n+1)/2}$
\item $w_{+,0} := 1/q^\partial_{+,0}$
\item $w_{-,0} := 1/q^\partial_{-,0}$
\end{itemize}
\leavevmode \\
\textbf{Steps to compute $c_k$}:
\begin{enumerate}
\item \label{t:qdj}For $j\in[1,k-1]$ calculate $q_j^\partial$ by
\begin{align*}
q_j^\partial(x,\xi,R)=&\sum_{\gamma\in I_j, \mathrm{rk}(\gamma)<(n+1)/2} \mathfrak{c}_{\mathrm{rk}(\gamma),n}C_{bdy}^{(\gamma)}(x,-D_\xi)(R^2+g_{bdy}(\xi,\xi))^{-(n+1)/2+\mathrm{rk}(\gamma)}+\\
&{ -}\sum_{\gamma\in I_j, \mathrm{rk}(\gamma)\geq (n+1)/2} \mathfrak{c}_{\mathrm{rk}(\gamma),n}C_{bdy}^{(\gamma)}(x,-D_\xi)\\
&\qquad \qquad \qquad \qquad \left[(R^2+g_{bdy}(\xi,\xi))^{-(n+1)/2+\mathrm{rk}(\gamma)}\log(R^2+g_{bdy}(\xi,\xi))\right],
\end{align*}
for $n$ odd, and for $n$ even by
\begin{align*}
q_j^\partial(x,\xi,R)=&\sum_{\gamma\in I_j} \mathfrak{c}_{\mathrm{rk}(\gamma),n}C_{bdy}^{(\gamma)}(x,-D_\xi)(R^2+g_{bdy}(\xi,\xi))^{-(n+1)/2+\mathrm{rk}(\gamma)}.
\end{align*}
\item \label{t:qpm}For $j\in[1,k-1]$ calculate 
\begin{align}
\frac{q^\partial_j}{q^\partial_0}-\frac{1}{q^\partial_0}\sum_{\substack{k+l+|\alpha|=j\\k,l<j}}\frac{1}{\alpha!} \partial_\xi^\alpha q_{+,k}^\partial D^\alpha_x q_{-,l}^\partial. \label{eq:pf}
\end{align}
\item \label{t:pf}Using a partial fraction decomposition, partition the terms of \eqref{eq:pf} into ones including factors of $(\xi_n-h_+)$ and $(\xi_n - h_-)$. Call the sum of the former $\mathfrak{q}_{+,j}$ and the latter $\mathfrak{q}_{+,j}$.
\item \label{t:pfok} Set $q_{\pm,j}^\partial := q_{\pm,0}^\partial\mathfrak{q}_{+,j}$
\item \label{t:w} For $j\in[1,k-1]$ calculate $w_{\pm,j}$ by 
$$w_{\pm,j}=-w_{\pm, 0}\sum_{k+l+|\alpha|=j, \, l<j}\frac{1}{\alpha!} \partial_\xi^\alpha q_{\pm, k}^\partial D_x^\alpha w_{\pm,l}.$$
\item \label{t:bdk}Calculate $B_{\rd^2,k}$ by
\begin{align*}
B_{\rd^2,k}(x'):=\sum_{k=|\beta|+j+l+1}\frac{(-i)^{|\beta|+1}}{\beta'!(\beta_n+1)!}\partial_{x}^{\beta}w_{-,j}(x',0,0,1) \partial_\xi^{\beta+(0,1)} w_{+,l}(x',0,0,1).
\end{align*}
For $k=0$ this is 0.
\item Calculate $c_k$ by 
$$c_{k}(M,\rd)=\int_X a_{k,0}(x,1)\rd x+\int_{\partial X}B_{d^2,k}(x)\rd x'.$$
\end{enumerate}

\end{example}

\begin{example}[$X=M\subset \R^N$ without boundary]

\leavevmode \\
\textbf{Input}:
\begin{itemize}
\item $n$: dimension of $M$
\item $N$: dimension of Euclidean space
\item $(x,\xi,R) \in \R^N\times \R^N\times{\R}$
\item $k$: $n-k$ is the degree of term to be determined
\item $\varphi_l(x_1,\ldots,x_n)$ for $l \in [n+1,N]$: parametrize near a point in $M$
\end{itemize}
\leavevmode \\
\textbf{Definitions}:
\begin{itemize}
\item $D_t = -i\partial_t$
\item $I_j:=\{\gamma\in \cup_{k=1}^\infty \N^k_{\geq 3}: |\gamma|=j+2k\}$
\item $\mathrm{rk}(\gamma) = $length$(\gamma)$
\item $$\mathfrak{c}_{m,n}:= 
\begin{cases}
(-1)^m(n-2m)!\omega_{n-2m}\omega_{2m}, \; &\mbox{for $2m- n-1<0$}\\
\frac{(-1)^{1-n/2} \omega_{2m}}{(2m-n)!\omega_{2m-n}}, \; & \mbox{for $2m-n-1\in 2\N+1$}\\
\frac{(-1)^{\frac{n+1}{2}}}{(2\pi)^{2m-n}}\omega_{2m}\omega_{2m-n-1} , \; &\mbox{for $2m- n-1\in 2\N$}
\end{cases},$$
\item $$C^j(x,v)=\sum_{l=n+1}^N\sum_{\substack{|\alpha|+|\beta|=j,\\ |\alpha|,|\beta|>0}} \frac{\partial_{x}^{\alpha}\varphi_l(x)\partial_{x}^{\beta}\varphi_l(x)}{\alpha!\beta!} v^{\alpha+\beta}.$$
\item $$q_{k,p}(x,R).v:=\sum_{|\alpha|=p}\partial_\xi^\alpha q_k(x,\xi,R)v^\alpha|_{\xi=0}=\partial_t^pq_k(x,tv,R)|_{t=0}.$$
\end{itemize}

\leavevmode \\
\textbf{Steps}:
\begin{enumerate}
\item For $j \in [1,k]$ calculate $q_j$ by 
\begin{align*}
q_j^\partial(x,\xi,R)=&\sum_{\gamma\in I_j, \mathrm{rk}(\gamma)<(n+1)/2} \mathfrak{c}_{\mathrm{rk}(\gamma),n}C_{bdy}^{(\gamma)}(x,-D_\xi)(R^2+g_{bdy}(\xi,\xi))^{-(n+1)/2+\mathrm{rk}(\gamma)}+\\
&{ -}\sum_{\gamma\in I_j, \mathrm{rk}(\gamma)\geq (n+1)/2} \mathfrak{c}_{\mathrm{rk}(\gamma),n}C_{bdy}^{(\gamma)}(x,-D_\xi)\\
& \qquad \qquad \qquad \qquad \left[(R^2+g_{bdy}(\xi,\xi))^{-(n+1)/2+\mathrm{rk}(\gamma)}\log(R^2+g_{bdy}(\xi,\xi))\right],
\end{align*}
for $n$ odd, and for $n$ even by
\begin{align*}
q_j^\partial(x,\xi,R)=&\sum_{\gamma\in I_j} \mathfrak{c}_{\mathrm{rk}(\gamma),n}C_{bdy}^{(\gamma)}(x,-D_\xi)(R^2+g_{bdy}(\xi,\xi))^{-(n+1)/2+\mathrm{rk}(\gamma)}.
\end{align*}
\item For $j \in [1,k]$ calculate $a_j$ by
\begin{align*}
a_{j,0}(x,R) = 
\begin{cases}
0& \mbox{$j $ odd},\\
-\frac{1}{n!\omega_n}R^{n+1}\sum_{\substack{k+2l+p=j\\2l<j, \, 2|k+p}}i^pq_{k,p}(x,R).\nabla_x^pa_{j-2k,0}(x,R) & \mbox{$j$ even}.
\end{cases}
\end{align*} 
\item Calculate $c_k$ by 
$$c_{k}(M,\rd)=\int_X a_{k,0}(x,1)\rd x.$$
\end{enumerate}

\end{example}

\section{Implementation of algorithm for 2-dimensional Euclidean domains}
\label{appecode}

The algorithm from Appendix B was implemented in Python, for the case of Euclidean domains. This was done with the use of the module \texttt{sympy}, a package used in python for symbolic manipulation, and the resulting code for $2$-dimensional domains is available in \cite{code}.

The algorithm computes the coefficient $c_k$ for arbitrary $k$, with two limitations: The dimension $n$ is fixed, and the computational effort and memory  required to compute $c_k$ increase rapidly with $k$. The main computational effort is the partial fraction decomposition, because of the large number of terms involved. 

We have used the code to compute $c_k$ up to $k=4$. We do not expect this code to be used beyond $k=5$, because of the rapid increase in memory required for subsequent $c_k$. As an illustration, the calculation of $q_{\pm,2}$ requires $6$ times more memory than $q_{\pm,1}$, and $q_{\pm,3}$ requires $40$ times more memory than $q_{\pm,2}$. Similarly, for $c_3$ the total number of terms requiring a partial fraction decomposition is approximately $1000$, whereas for $c_4$ this number increases to approximately $110000$.

\end{appendix}

\end{document}